\documentclass[3p]{amsart}
\usepackage{a4wide}
\usepackage{epsfig,amsmath,amsfonts,amssymb,amsthm,rotating,mathtools}
\usepackage{tikz,pgfplots}
\usetikzlibrary{arrows,snakes,backgrounds}
\usepackage{comment}
\usepackage{algorithm}			
\usepackage{algpseudocode}
\usepackage{mathrsfs}
\usepackage{placeins}
\usepackage{hyperref}
\usepackage{bbm}
\usepackage{pgfplots}
\usetikzlibrary{cd}
\pgfplotsset{compat=newest}
\pgfplotsset{plot coordinates/math parser=false}
\newlength\figureheight
\newlength\figurewidth 

\newcommand{\E}{\mathbb{E}}

\renewcommand{\P}{\mathbb{P}}
\newcommand{\R}{\mathbb{R}}

\newcommand{\bb}{\pmb}
\newcommand{\bfa}{\mathbf{a}}
\newcommand{\bfb}{\mathbf{b}}
\newcommand{\bfc}{\mathbf{c}}

\newcommand{\bfg}{\mathbf{g}}
\newcommand{\bfh}{\mathbf{h}}

\newcommand{\bfj}{\mathbf{j}}

\newcommand{\bfn}{\mathbf{n}}

\newcommand{\bfp}{\mathbf{p}}

\newcommand{\bfu}{\mathbf{u}}
\newcommand{\bfv}{\mathbf{v}}

\newcommand{\bfx}{\mathbf{x}}
\newcommand{\bfy}{\mathbf{y}}

\newcommand{\bfA}{\mathbf{A}}
\newcommand{\bfB}{\mathbf{B}}
\newcommand{\bfC}{\mathbf{C}}

\newcommand{\bfE}{\mathbf{E}}

\newcommand{\bfH}{\mathbf{H}}
\newcommand{\bfI}{\mathbf{I}}

\newcommand{\bfL}{\mathbf{L}}
\newcommand{\bfM}{\mathbf{M}}
\newcommand{\bfN}{\mathbf{N}}

\newcommand{\bfR}{\mathbf{R}}
\newcommand{\bfS}{\mathbf{S}}

\newcommand{\bfV}{\mathbf{V}}

\newcommand{\bfNull}{\mathbf{0}}
\newcommand{\bfga}{\boldsymbol{\gamma}}
\newcommand{\bfpsi}{\boldsymbol{\psi}}

\newcommand{\bfPi}{\boldsymbol{\Pi}}
\newcommand{\lVnr}{\langle\bfV,\bfn\rangle}

\newcommand{\cH}{\mathcal{H}}

\newcommand{\cK}{\mathcal{K}}
\newcommand{\cO}{\mathcal{O}}

\newcommand{\bfcC}{\boldsymbol{\mathcal{C}}}
\newcommand{\bfcS}{\boldsymbol{\mathcal{S}}}

\renewcommand{\div}{\operatorname{div}}
\DeclareMathOperator{\curl}{{curl}}
\DeclareMathOperator{\bcurl}{{\mathbf{curl}}}
\DeclareMathOperator{\bgrad}{{\boldsymbol{\nabla}}}
\DeclareMathOperator{\bDelta}{\boldsymbol{\Delta}}
\newcommand{\bcurlg}{\bcurl_{\Gamma}}

\newcommand{\isdef}{\mathrel{\mathrel{\mathop:}=}}

\newcommand{\Cor}{\operatorname{Cor}}
\newcommand{\Cov}{\operatorname{Cov}}

\newcommand{\diag}{\operatorname{diag}}

\renewcommand{\div}{\operatorname{div}}

\renewcommand{\d}{\operatorname{d}\!}

\newcommand{\tr}{\operatorname{tr}}

\DeclareMathOperator{\Id}{\operatorname{Id}}
\DeclareMathOperator{\tvec}{\operatorname{vec}}

\newcommand\numberthis{\addtocounter{equation}{1}\tag{\theequation}}
\theoremstyle{plain}             
\newtheorem{theorem}{Theorem}[section]
\newtheorem{lemma}[theorem]{Lemma}
\newtheorem{corollary}[theorem]{Corollary}

\newcommand{\loc}{{\mathrm{loc}}}

\newcommand{\chispace}{ {\bfH^{-1/2}_\times(\div_\Gamma,\Gamma)}{}}
\newcommand{\chispacecurl}{ {\bfH^{-1/2}_\times(\curl_\Gamma,\Gamma)}{}}

\newcommand{\halfmixchispace}{ {\bfH^{-1/2}_{0,\times}(\div_\Gamma,\Gamma)}{}}

\begin{document}
\title[A perturbation approach for electromagnetic
problems on random domains]
{A higher order perturbation approach for electromagnetic
scattering problems on random domains}

\author{J\"urgen D\"olz}
\address{TU Darmstadt, Department of Mathematics, Dolivostrasse 15, 64293 Darmstadt, Germany}
\email{doelz@mathematik.tu-darmstadt.de}
\thanks{The author was partially supported by SNSF Grant 174987, as well as the Graduate School of Computational Engineering at TU Darmstadt, the Excellence Initiative of the German Federal and State Governments, and the Graduate School of Computational Engineering at TU Darmstadt. }

\begin{abstract}
We consider time-harmonic electromagnetic scattering problems on
perfectly conducting scatterers with uncertain shape. Thus, the scattered field will also be uncertain. Based on the knowledge of the two-point correlation of the domain boundary variations around a reference domain, we derive a perturbation analysis for the mean of the scattered field. Therefore, we compute the second shape derivative of the scattering problem for a single perturbation. Taking the mean, this leads to an at least third order accurate approximation with respect to the perturbation amplitude of the domain variations. To compute the required second order correction term, a tensor product equation on the domain boundary has to be solved. We discuss its discretization and efficient solution using boundary integral equations. Numerical experiments in three dimensions are presented.
\end{abstract}

\maketitle


\section{Introduction}

Electromagnetic scattering is a phenomenon which occurs in various areas of interest
in engineering such as radar scattering, simulation of wireless networks, etc. Due to
its physical complexity, computer simulations are inevitable during
engineering projects for the prediction of the behaviour of scattered electromagnetic waves.
In many cases, the modeling and simulation of electromagnetic scattering
problems is well understood for full and deterministic knowledge of the input data.
However, due to manufacturing tolerances, the manufactured objects are unlikely to
coincide with their mathematical model. This can be interpreted as a random behaviour
in their realization and has raised considerable interest in the engineering
community, see, e.g., \cite{GLRS18, LYB+19, SGB+19, SAJF18}.

The present article shall deal with the particular problem of electromagnetic scattering
on random, perfect conductors in the time-harmonic regime. When formulated in terms of
partial differential equations (PDEs), this amounts to the solution of such an equation on a
random domain. A crucial role in the numerical treatment of this kind of problems is
the mathematical modeling of the randomness in the domain. Common approaches in the
literature can be split into domain mapping approaches, see \cite{CNT16, HPS16, JSZ17,XT06},
and perturbation approaches, see \cite{DH18, HSS08, JS16}. A hybrid approach combining the
two methods was presented in \cite{CNT17}. Domain mapping approaches transfer
the uncertainty in the domain onto a PDE with random coefficients
on a fixed reference domain and are able to deal with large deformations. For the computation
of statistical quantities of interest such as the mean or the variance of the solutions,
high-dimensional quadrature rules such as (quasi) Monte-Carlo or sparse grids are usually
employed, see, e.g., \cite{BG04b, Caf98, DKS13}.

However, for small domain
perturbations, as they occur for example in optical gratings, see \cite{GLRS18,SAJF18}, perturbation approaches usually provide a more efficient tool for an
accurate approximation of the mean or the variance of the PDE's solution.
The main idea is to model the randomness in the domains as a perturbation of size
$\varepsilon$ of a reference domain.
In the context of time-harmonic electromagnetic scattering problems, one can show that under some
smoothness assumptions on the domain, see \cite{JS16}, the mean of the random scattered
waves $\bfE_\varepsilon^s(\omega)$ can be expressed as
\begin{align}\label{eq:meanfirstorder}
\E\big[\bfE_\varepsilon^s\big]=\bfE^s_0+\cO(\varepsilon^2),\qquad 0<\varepsilon\leq\varepsilon_0,
\end{align}
in some domain which has a positive distance from all
random domains. Here, $\bfE_0^s$ is the scattered wave of the reference domain,
which is, without loss of generality, assumed to be the mean of all random domains.

The contribution of the present article is as follows. First, we characterize the second
order correction term of a single random domain perturbation $\bfV(\omega)$. I.e., we give a characterization of
the term $\delta^2\bfE^s[\bfV(\omega),
\bfV(\omega)]$ in the shape Taylor expansion around the unperturbed problem
\[
\bfE^s(\omega) = \bfE^s_0+\varepsilon\delta\bfE^s[\bfV(\omega)]+\frac{\varepsilon^2}{2}\delta^2\bfE^s[\bfV(\omega),
\bfV(\omega)]
+\cO(\varepsilon^3)
\]
of the scattered field. To the best of the author's knowledge, the previous works on domain
derivatives for time-harmonic scattering \cite{CLL12, CLL12b, Het12, HL18, Pot96} only
characterize the first order correction term $\delta\bfE^s[\bfV(\omega)]$. Based on this second
order correction term we improve the expansion \eqref{eq:meanfirstorder}
by an additional second order correction term to
\[
\E\big[\bfE_\varepsilon^s\big]=\bfE^s_0+\frac{\varepsilon^2}{2}
\E\big[\delta^2\bfE^s\big]+\cO(\varepsilon^3).
\]
As we will see, this expansion becomes even
$\cO(\varepsilon^4)$ accurate if the distribution of the random domain satisfies
some symmetry conditions. The second order correction term $\E\big[\delta^2\bfE^s\big]$ can
be characterized as the solution of an electromagnetic scattering problem with certain
boundary data. These boundary data depend on statistical quantities depending on the
random domain perturbations.

Unfortunately, the random domain perturbations enter these statistical quantities in a non-linear way. To circumvent this obstacle, we follow the
approach of \cite{DH18} and reformulate
these non-linearities as diagonals of correlations of random variables on the boundary of
the reference domain. Using boundary integral equations, which are well established for
the solution of electromagnetic scattering problems, see \cite{BH03, CK12, Ned01}, we show that these correlations can
be obtained from the solution of a single additional equation in a tensor product space.
We discuss the solution of this equation by the means of a Galerkin scheme. The introduction
of suitable $L^2$-projections allows the use of any standard technique for the efficient treatment of correlation equations to accelerate the computations.

The remainder of the article is structured as follows. Section~\ref{sec:prel} recalls
the required terminology for the treatment of electromagnetic scattering problems.
In section~\ref{sec:randomdomains} we state the electromagnetic scattering problem on random
domains and compute the second Fr\'echet derivative of the scattered electromagnetic wave for
a single domain perturbation. Section~\ref{sec:BIE} is concerned with the derivation of
the second order correction term for the mean of the scattered field and its boundary values,
whereas section~\ref{sec:galerkin} is concerned with the Galerkin discretization. In
section~\ref{sec:fast}, we comment on efficient solution techniques of the corresponding
tensor product equation, whereas section~\ref{sec:examples} is concerned with numerical
examples. Finally, in section~\ref{sec:concl}, we draw our conclusions.

\section{Preliminaries}
\label{sec:prel}

On a bounded Lipschitz domain $D\subset\mathbb{R}^3$ and for $0\leq s$ we denote by $H^s(D)$ the usual Sobolev spaces \cite{McL01}, and by $\mathbf{H}^s(D)$ their vector valued counterparts. For $s=0$ we use the convention $H (D) = H^0 (D) = L^2(D)$ and $\mathbf{H} (D) = \mathbf{H}^0 (D) = \mathbf{L}^2(D)$. On unbounded domains $D^c:=\mathbb{R}^3\setminus\overline{D}$ we use the same notation together with the subscript $\loc$ in the form of $H^s_\loc(D^c)$ and $\mathbf{H}^s_\loc(D^c)$ to denote that the required integrability conditions must only be fulfilled on all bounded subdomains.

For compact manifolds $\Gamma=\partial D$ we denote by $H^s(\Gamma)$ the usual construction of Sobolev spaces on manifolds via charts and by $\mathbf{H}^s(\Gamma)$ their vector valued counterparts. As commonly done, we define the spaces $H^{-s}(\Gamma)$ and $\mathbf{H}^{-s}(\Gamma)$ as the dual spaces of $H^s(\Gamma)$ and $\mathbf{H}^s(\Gamma)$ with $L^2(\Gamma)$ and $\mathbf{L}^2(\Gamma)$ as pivot spaces. The corresponding duality
products are denoted by $\langle\cdot,\cdot\rangle_0$ and
$\langle\cdot,\cdot\rangle_{\bfNull}$.

Let $\mathcal{M}$ be one of the domains $D$, $D^c$, or their boundary $\Gamma$. For any differential operators $\operatorname{d}$ defined on $\mathcal{M}$, we define the spaces $H^s(\operatorname{d},\mathcal{M})$ via the closure of $H^s(\mathcal{M})$ under the graph norm $\|\cdot\|_{H^s(\mathcal{M})}+\|\operatorname d(\cdot)\|_{H^s(\mathcal{M})}$, equipping the spaces with the same. This definition is generalised to vector valued differential operators and spaces in complete analogy.

For further reference we introduce the following surface differential operators, see also \cite[Chapter 2.5.6]{Ned01}. Given a sufficiently smooth
scalar valued function $u$ on $\Gamma$ with a sufficiently smooth extension $\tilde{u}$ into $D^c$, the vector valued surface curl $\bcurl_{\Gamma}$ and the surface gradient
$\bgrad_{\Gamma}$ are defined as
\begin{align*}
\bcurl_{\Gamma} u={}&\big[\bgrad\tilde{u}\big]\times\bfn,\\
\bgrad_{\Gamma} u={}&\bfn\times\bcurl_{\Gamma}u.
\end{align*}
The surface differential operators to not depend on the chosen extension $\tilde{u}$. By $\bfn$ we denote the exterior normal of $\Gamma$. For vector valued functions $\bfu$,
we denote by $\big[\bgrad_{\Gamma}\bfu\big]$ the matrix whose $i$-th column contains the surface gradient of the $i$-th component of $\bfu$.

Given a tangential vector field $\bfu$ to $\Gamma$ and a sufficiently smooth extension $\tilde{\bfu}$, we define the scalar valued surface curl
$\curl_\Gamma$ and the surface divergence $\div_\Gamma$ by
\begin{align*}
\curl_\Gamma\bfu={}&\langle\bcurl\tilde{\bfu},\bfn\rangle,\\
\div_\Gamma\bfu={}&\curl_\Gamma(\bfn\times\bfu).
\end{align*}
Here, we denote by $\langle\cdot,\cdot\rangle$ the Euclidean scalar product in $\R^3$. Again, the surface differential operators do not depend
on the chosen extension $\tilde{\bfu}$.

We also require the following tangential trace operators, see also \cite{BC03}.	For $\mathbf{u}\in C(D^c; \mathbb C^3)$, we define the exterior rotated tangential trace operator as
\begin{align*}
\bfga_{ t}^+ (\mathbf{u})(\mathbf{x}_0) & = \lim_{\substack{\mathbf{x}\to \mathbf{x}_0\\\mathbf{x}\in D^c}}\bfn_{\mathbf{x}_0}\times\mathbf{u}(\mathbf{x}),\quad\text{for all}~\mathbf{x}_0\in\Gamma,
\end{align*}
where $\bfn_{\mathbf{x}_0}$ denotes the exterior normal vector of $\Gamma$ at $\mathbf{x}_0$. Further, the magnetic rotated tangential trace and the
tangential trace are defined as
\begin{align*}
\bfga_N^+={}&\bfga_{t}^+\circ\bcurl,\\
\bfga_0^+={}&\bfga_t^+(\cdot)\times\bfn.
\end{align*}
We add the ``+'' to our notation to stress the fact that we consider the trace from the exterior domain only.
By density arguments, see also \cite{BC03}, these traces can be extended to be continuous operators on $\mathbf{H}_\loc(\bcurl,D^c)$, where we define
\begin{align*}
\chispace ={}& \bfga_t^+\big(\mathbf{H}_\loc(\bcurl,D^c)\big),\\
\chispacecurl ={}& \bfga_0^+\big(\mathbf{H}_\loc(\bcurl,D^c)\big).
\end{align*}
The space $\chispace$ is a Hilbert space and its own dual with respect to the pairing
\[
\langle \bb\mu,\bb\nu\rangle_{\times} = \int_\Gamma (\bb\mu\times\bfn_{\mathbf{x}}) \cdot \bb\nu \d \Gamma_{\mathbf{x}}.
\]
In fact, the introduced surface differential and trace operators can be extended to continuous operators on Sobolev spaces such that the
diagram
\begin{equation}\label{eq:diagram}
\begin{tikzcd}[row sep = 2em,column sep = 1.3cm]
H^1(D^c)\ar{dd}[description]{\cdot|_{\Gamma}}\ar{r}[description]{\bgrad}&
\bfH(\bcurl,D^c)\ar{r}[description]{\bcurl}\ar{d}[description]{\bfga_0^+} \arrow[ddd,bend right=60,"{\bfga_t^+}" description, near start]&
\bfH(\div,D^c)\ar{dd}[description]{\langle\cdot,\bfn\rangle}\\
&\bfH_\times^{{-1/2}}(\curl_\Gamma,\Gamma)\ar{dr}[description]{\curl_\Gamma}  \arrow{dd}[description]{\cdot\times\bfn} &\\
H^{{1/2}}(\Gamma)\ar{dr}[description]{{\bcurl_\Gamma}} \ar{ur}[description]{\bgrad_\Gamma} & &             H^{{-1/2}}(\Gamma)\\
&\bfH_\times^{{-1/2}}(\div_\Gamma,\Gamma)\ar{ur}[description]{\div_\Gamma}&
\end{tikzcd}
\end{equation}
commutes, see also \cite{BDK+18}.

For further reference we also introduce the tensor product Sobolev space on $\Gamma\times\Gamma$ defined by
\begin{align*}
\halfmixchispace={}&L^2(\Gamma)\otimes\chispace.
\end{align*}
By a tensor product argument, we see that this space is self-dual with
respect to the duality product $\langle\cdot,\cdot\rangle_{0,\times}$
which is defined on simple tensors as
\begin{align*}
\langle u_1\otimes\bfu_2,v_1\otimes\bfv_2\rangle_{0,\times}
={}&
\langle u_1, v_1\rangle_0\langle\bfu_2,\bfv_2\rangle_\times
\end{align*}
for $u_1,v_1\in L^2(\Gamma)$ and $\bfu_2,\bfv_2\in\chispace$.

\section{Perturbation Analysis for Electromagnetic Scattering Problems}
\label{sec:randomdomains}

\subsection{The Electromagnetic Scattering Problem on Random Domains}

Given a perfectly conducting object $D\subset\R^3$ with Lipschitz boundary $\Gamma$ in a surrounding $D^c$, we are interested in the scattered field $\bfE^s$ of an electric incident wave $\bfE^i$ hitting the scatterer $D$. Assuming a time-harmonic problem, the scattered field $\bfE^s$ can then be described by the \emph{electric wave equation}
\begin{align}\label{problem::ext_scattering}
\begin{aligned}
	\bcurl \bcurl\, \bfE^s -\kappa^2\bfE^s &=\mathbf{0}&&\text{in}~D^c,\\
	\bfga_t^+\bfE^s &= -\bfga_t^+\bfE^i&&\text{on}~\Gamma,\\
	\big|\bcurl\,\bfE^s(\mathbf{x})\times{\mathbf{x}}\cdot{|\mathbf{x}|}^{-1}-i\kappa\bfE^s(\mathbf{x})\big|&=\mathcal{O}(|\mathbf{x}|^{-2})&&\text{for}~|\mathbf{x}|\to\infty.
\end{aligned}
\end{align}
The last equation is referred to as Silver-M\"uller radiation condition.
It is well known that \eqref{problem::ext_scattering} is uniquely solvable for any sufficiently regular Dirichlet data and wavenumber $\kappa >0$, see, e.g., \cite{BH03}.
Given an incident wave $\bfE^i$, the total electric field $\bfE_{\mathrm{tot}}$ in $D^c$ is then given by $\bfE_{\mathrm{tot}} = \bfE^i+\bfE^s$.

For the formulation of the scattering problem on random domains,
let $D_0\subset\R^3$ be a reference 
domain with, a $C^{2,1}$-boundary. On a separable, complete probability space $(\Omega,
\Sigma,\P)$, consider a random vector field $\bfV\in L_\P^2(\Omega; C^{2,1}
(\Gamma,\R^3))$ with $\|\bfV(\omega,\cdot)\| _{C^{2,1}(\Gamma;\mathbb{R}^3)}
\lesssim 1$ uniformly for all $\omega\in\Omega$. The random vector field perturbs the boundary 
of the reference domain $\Gamma$ in accordance with $\partial D_{\varepsilon}
(\omega)\isdef\Gamma+\varepsilon\bfV(\omega ,\Gamma)$ for some 
given $\varepsilon>0$. A \emph{random domain} $D_{\varepsilon}(\omega)$ 
is then given by the interior of the perturbed boundary $\partial D_{\varepsilon}
(\omega)$. For later considerations, we also introduce a compact set $G\subset D^c$,
which we assume to have an arbitrary positive distance to
\[
D_{\varepsilon}^{\cup\Omega}\isdef \bigcup _{\omega\in\Omega}D_{\varepsilon}(\omega),\qquad 0<\varepsilon\leq\varepsilon_0.
\]

Note that, in contrast to the domain mapping approach which requires 
a vector field on the whole reference domain, the perturbation approach 
only requires a vector field on the boundary. A correspondence between 
the two approaches is given by the fact that every vector field on the 
boundary $\Gamma$ can smoothly be extended into the exterior
$D^c$ in such a way that it vanishes on the compactum $G$.

Given an incident wave $\bfE^i$, the scattering problem on the introduced
random domains $D_{\varepsilon}^c(\omega)$ reads
\begin{align}\label{eq:randmaxwell}
\begin{aligned}
\bcurl\bcurl \bfE_{\varepsilon}^s(\omega)-\kappa^2\bfE_{\varepsilon}^s(\omega) ={}&\textbf{0}\quad\text{in}~D_{\varepsilon}^c(\omega),\\
\bfga_t^+\bfE_{\varepsilon}^s(\omega) ={}&\bfg\quad\text{on}~\partial D_{\varepsilon}^c(\omega),
\end{aligned}
\end{align}
with $\bfg=-\bfga_t^+\bfE^i$ and complemented by Silver-M\"uller radiation conditions.
We note the dependence of the scattered
wave on $\varepsilon$ and $\omega$ and that the tangential trace of the incident
wave is taken onto the perturbed boundary $\partial D_{\varepsilon}^c(\omega)$.

A special case is the solution
of the scattering problem on the unperturbed reference domain $D_0^c$, which is
given by
\begin{align}\label{eq:refmaxwell}
\begin{aligned}
\bcurl\bcurl \bfE_0^s-\kappa^2\bfE_0^s ={}&\mathbf{0}\quad\text{in}~D_0^c,\\
\bfga_t^+\bfE_0^s ={}&\bfg\quad\text{on}~\Gamma,
\end{aligned}
\end{align}
and complemented by Silver-M\"uller radiation conditions.
We denote the total electric field around $D_0^c$ by $\bfE_0=\bfE_0^i+\bfE_0^s$.

\subsection{Shape Calculus for Parametrized Domains}
To deal with the non-linear dependence of $\bfE^s_{\varepsilon}(\omega)$ on
$D_{\varepsilon}^c(\omega)$, we may exploit that the dependence is
Fr\'echet-differentiable, see \cite{CLL12, CLL12b, Het12, HL18, Pot96}. More precisely,
given the reference domain $D_0$ and the vector field $\bfV\in
L_\mathbb{P}^2(\Omega; C^{2,1}(\Gamma;\mathbb{R}^3))$ with
$\|\bfV(\omega,\cdot)\| _{C^{2,1}(\Gamma;\mathbb{R}^3)}
\lesssim 1$ uniformly for all $\omega\in\Omega$, one can expand
$\bfE^s_{\varepsilon}(\omega)$ into a \emph{shape Taylor expansion}
\begin{align}\label{eq:shapeTaylor}
\bfE^s_{\varepsilon}(\omega, \mathbf{x}) = \bfE^s_0(\mathbf{x})+\varepsilon\delta\bfE^s[\bfV
(\omega)](\mathbf{x})+\frac{\varepsilon^2}{2}\delta^2\bfE^s[\bfV(\omega),
\bfV(\omega)](\mathbf{x})+\mathbf{R}_2(\varepsilon\bfV(\omega))(\mathbf{x}),\quad\mathbf{x}\in G,
\end{align}
which holds for all $0<\varepsilon\leq\varepsilon_0$ for some 
$\varepsilon_0 >0$ small enough. Here, the \emph{first order
local shape derivative} $\delta\bfE^s[\bfV(\omega)]$ is given as the
solution of
\begin{align}\label{eq:firstlocaldu}
\begin{aligned}
\big(\bcurl\bcurl-\kappa^2\big)\delta\bfE^s[\bfV(\omega)] ={}& \textbf{0}&&\text{in}~D_0^c,\\
\bfga_t^+\delta\bfE^s[\bfV(\omega)] ={}&\bfg^{(\delta,s)}[\bfV(\omega)]&&\text{on}~\Gamma,
\end{aligned}
\end{align}
satisfying the Silver-M\"uller radiation conditions and suitable boundary conditions $\bfg^{(\delta,s)}[\bfV(\omega)]$ which we
will discuss in the next subsection.
Given a second vector field $\bfV'\in L_\P^2(\Omega;C^{2,1}(\Gamma;\R^3))$
for which it holds $\|\bfV'(\omega,\cdot)\| _{C^{2,1}(\Gamma;\R^3)}\lesssim 1$
uniformly for all $\omega\in\Omega$, we will prove that the \emph{second order local shape
derivative} $\delta^2\bfE[\bfV(\omega),\mathbf{V'}(\omega)]$ is given as the solution of
\begin{align}\label{eq:secondlocaldu}
\begin{aligned}
\big(\bcurl\bcurl-\kappa^2\big)\delta^2\bfE^s[\bfV(\omega),\mathbf{V'}(\omega)]
={}&\textbf{0}&&\text{in}~D_0,\\
\delta^2\bfE^s[\bfV(\omega),\mathbf{V'}(\omega)]
={}&\bfg^{(\delta^2,s)}[\bfV(\omega),\bfV'(\omega)]&&\text{on}~\Gamma,
\end{aligned}
\end{align}
with certain boundary conditions $\bfg^{(\delta^2,s)}[\bfV(\omega),\bfV'(\omega)]$
and Silver-M\"uller radiation conditions towards infinity.

By definition of the Fr\'echet derivative, the
remainder $\bfR_2$ in \eqref{eq:shapeTaylor} is uniformly in
$\varepsilon\bfV(\omega)$ negligible with respect to
$\|\varepsilon\bfV(\omega)\|_{C^{3,1}(\Gamma;\R^3)}^2$, i.e.,
$\mathbf{R}_2(\varepsilon\bfV(\omega))=o\big(\|\varepsilon\bfV(\omega)\|_{C^{3,1}(\Gamma;\R^3)}^2\big)$. To arrive at the asymptotic
\begin{align}\label{eq:correctionterm}
\mathbf{R}_2(\varepsilon\bfV(\omega))
=
\mathcal{O}(\varepsilon^3),
\end{align}
we have to assume that the second order correction term satisfies a Lipschitz condition.
This was proven for scalar valued cases in \cite{Dam02}, but, to the best of the author's
knowledge, is unproven for the present case. Together 
with the assumption that $\|\bfV(\omega,\cdot)\|
_{C^{3,1}(\Gamma;\mathbb{R}^n)}\lesssim 1$ uniformly for all
$\omega\in\Omega$, the Lipschitz assumption yields that the constant in \eqref{eq:correctionterm} is independent
of $\omega$. Numerical experiments indicate that this Lipschitz condition is
likely to be satisfied.

\subsection{The Second Local Shape Derivative}
Infinite Fr\'echet differentiability of the scattered electromagnetic field was proven
first in \cite[Theorem 6]{Pot96}. The following characterization is \cite[Theorem 7]{Pot96},
where we write
\[
\bfB^{2,1}_{\varepsilon}(\Gamma)\coloneqq\big\{\bfV\in C^{2,1}(\Gamma;\R^3)\colon\|\bfV\|_{C^{2,1}(\Gamma;\R^3)}<\varepsilon\big\}.
\]
\begin{theorem}\label{thm:potthastthm}
Let $\Gamma$ be a boundary of class $C^{2,1}$. Consider for each
$\bfV\in\bfB^{2,1}_{\varepsilon}(\Gamma)$
the solution $\bfE^s[\bfV]$ to the exterior
problem \eqref{problem::ext_scattering} for the domain $D[\bfV]$ given by
$\partial D[\bfV]=\Gamma+\bfV$ with boundary values
$\bfg[\bfV]\in CT^{1,1}(\partial D[\bfV];\R^3)$, i.e., $\bfE^s[\bfV]$ solves the
electric wave equation in the open exterior of $D[\bfV]$, satisfies the
Silver-M\"uller radiation condition and the boundary condition
$\bfn[\bfV]\times\bfE^s[\bfV]=\bfg[\bfV]$ on $\partial D[\bfV]$. Assume that
the mapping
\[
\bfB^{2,1}_{\varepsilon}(\Gamma)
\to C^{1,1}(\Gamma),
\quad
\bfh\to\bfg[\bfh]
\]
is Fr\'echet differentiable. Then the mapping
\[
\bfB^{2,1}_{\varepsilon}(\Gamma)\to C(G),
\quad
\bfh\to\bfE^s[\bfh]\]
is Fr\'echet differentiable. The Fr\'echet derivative $\delta\bfE^s[\bfV]$
in the point $\bfh=\mathbf{0}$ solves the exterior scattering problem
\eqref{problem::ext_scattering} with boundary values
\[
\bfn\times\delta\bfE^s[\bfV]=\bfg^{(\delta,s)}[\bfV],
\]
given by
\[
\bfg^{(\delta,s)}[\bfV]
=
\bfga_0^+\big(\delta\bfg[\bfV]-\delta\bfn[\bfV]\times\bfE^s\big)
-
\bfn\times\frac{\partial\bfE^s}{\partial\bfV}\qquad\text{on}~\Gamma.
\]
\end{theorem}
The first and the second Fr\'echet derivative of $\bfn$ can be characterized as follows.
\begin{lemma}\cite[Lemma~4.3]{CLL12}\label{lem:nder}
For the first and second Fr\'echet derivative of $\bfn$ it holds
\begin{align*}
\delta\bfn[\bfV]
=&{}
-(\nabla_{\Gamma}\bfV)\bfn,\\
\delta^2\bfn[\bfV,\bfV']
=&{}
(\nabla_{\Gamma}\bfV')(\nabla_{\Gamma}\bfV)\bfn
+(\nabla_{\Gamma}\bfV)(\nabla_{\Gamma}\bfV')\bfn
-\big\langle(\nabla_{\Gamma}\bfV)\bfn,(\nabla_{\Gamma}\bfV')\bfn\big\rangle\bfn.
\end{align*}
\end{lemma}

A consequence of theorem \ref{thm:potthastthm} and this lemma is the following characterization
of the first local shape derivative for the electromagnetic scattering
problem, see \cite[Corollary 8]{Pot96} and \cite[Theorem 6.6]{CLL12b}.
\begin{theorem}\label{thm:firstlocal}
	Let $\Gamma$ be a boundary of class $C^{2,1}$. The
	first local shape derivative of \eqref{problem::ext_scattering} in direction
	$\bfV\in C^{2,1}(\Gamma;\R^3)$ is the solution of \eqref{eq:firstlocaldu}
	with
	\[
	\bfg^{(\delta,s)}[\bfV]
	=
	-
	\delta\bfn[\bfV]\times\bfE_0
	-
	\bfn\times\frac{\partial\bfE_0}{\partial\bfV},
	\]
	where $\bfE_0=\bfE^i+\bfE^s_0$ is the solution to the original scattering problem \eqref{eq:refmaxwell} on the domain
	$D_0^c$. Moreover, there holds the alternate representation
	\[
	\bfg^{(\delta,s)}[\bfV]
	=
	-
	\langle\bfV,\bfn\rangle\bfga_{N}^{+}\bfE_0\times\bfn
	+
	\bcurlg\big(\langle\bfV,\bfn\rangle\langle\bfE_0,\bfn\rangle\big).
	\]
\end{theorem}
\begin{proof}
	The result is obtained by setting
	\[
	\bfg[\bfV](\bfx)=-\bfn[\bfV](\bfx)\times\bfE^i(\bfx+\bfV(\bfx)),\quad\bfx\in\Gamma,
	\]
	in theorem~\ref{thm:potthastthm}. The operator $\bfga_0^+$ acts as the identity, since
	$\bfE_0=\langle\bfE_0,\bfn\rangle\bfn$ and thus the normal components vanish. The alternate
	representation is proven in \cite[Theorem 6.6]{CLL12b}.
\end{proof}
We remark that due to
\[
\bcurl\bcurl\bfE=\kappa^2\bfE
\]
and a diagram chase in \eqref{eq:diagram} it follows that it holds
\begin{align}\label{eq:Enrep}
\langle\bfE,\bfn\rangle
=
-\kappa^{-2}\div_{\Gamma}(\bfga_{N}^{+}\bfE)
\qquad\text{on}~\Gamma
\end{align}
for all solutions $\bfE$ of the above electromagnetic scattering problems.
Thus, $\langle\bfE_0,\bfn\rangle$ can be computed from $\bfga_{N}^+\bfE_0$,
which is accessible when using numerical methods.

\begin{theorem}\label{cor:secder}
Let $\Gamma$ be a boundary of class $C^{3,1}$.
Then, the boundary values of the second derivative \eqref{eq:secondlocaldu} in directions
$\bfV,\bfV'\in C^{3,1}(\Gamma;\R^3)$ are given by
\begin{align}\label{eq:secondgeneral}
\begin{aligned}
\bfg^{(\delta^2,s)}[\bfV,\bfV']
={}&
-\delta^2\bfn[\bfV,\bfV']\times\bfE_0
-\delta\bfn[\bfV]\times\delta\bfE^s[\bfV']
-\delta\bfn[\bfV']\times\delta\bfE^s[\bfV]
-\delta\bfn[\bfV]\times\frac{\partial\bfE_0}{\partial\bfV'}\\
&\qquad
-\delta\bfn[\bfV']\times\frac{\partial\bfE_0}{\partial\bfV}
-\bfn\times\frac{\partial\big(\delta\bfE^s[\bfV']\big)}{\partial\bfV}
-\bfn\times\frac{\partial\big(\delta\bfE^s[\bfV]\big)}{\partial\bfV'}
-\bfn\times\frac{\partial^2\bfE_0}{\partial\bfV'\partial\bfV}.
\end{aligned}
\end{align}
\end{theorem}
\begin{proof}
The proof is a repeated application of theorem \ref{thm:potthastthm}. Since the second local shape
derivative is defined as the derivative of the first local shape derivative,
the theorem immediately yields
\[
\bfg^{(\delta^2,s)}[\bfV,\bfV']
=
\bfga_0^+\Big(\delta'\big(\bfg^{(\delta,s)}[\bfV]\big)[\bfV']-\delta\bfn[\bfV']\times\delta\bfE^s[\bfV]\Big)
-
\bfn\times\frac{\partial\big(\delta\bfE^s[\bfV]\big)}{\partial\bfV'},
\]
where we write $\delta'$ to indicate the derivative which is applied in the direction
$\bfV'$. We thus have to compute the derivative of
\[
\bfh'\mapsto\big(\bfg^{(\delta,s)}[\bfV]\big)[\bfh']
\]
at $\bfh'=\mathbf{0}$, with
\[
\big(\bfg^{(\delta,s)}[\bfV]\big)[\bfh']
=
-
\big(\delta\bfn[\bfV]\big)[\bfh']\times\bfE_0[\bfh']
-
\bfn[\bfh']\times\frac{\partial\bfE_0}{\partial\bfV}[\bfh']
\]
which is given by
\begin{align*}
\delta'\big(\bfg^{(\delta,s)}[\bfV]\big)[\bfV']
={}&
-\delta^2[\bfV,\bfV']\times\bfE_0
-\delta\bfn[\bfV]\times\delta\bfE_0[\bfV']
-\delta\bfn[\bfV]\times\frac{\partial\bfE_0}{\partial\bfV'}\\
&\qquad
-\delta\bfn[\bfV']\times\frac{\partial\bfE_0}{\partial\bfV}
-\bfn\times\frac{\partial\big(\delta\bfE^s[\bfV']\big)}{\partial\bfV}
-\bfn\times\frac{\partial^2\bfE_0}{\partial\bfV'\partial\bfV}.
\end{align*}
Putting everything together, exploiting that $\delta\bfE^i[\bfV']=\bfNull$ and that the tangential component of $\bfE_0$ vanishes
yields
\begin{align*}
\bfg^{(\delta^2,s)}[\bfV,\bfV']
={}&
-\delta^2\bfn[\bfV,\bfV']\times\bfE_0\\
&\qquad-\bfga_0^+\bigg(
\delta\bfn[\bfV]\times\delta\bfE^s[\bfV']
+\delta\bfn[\bfV']\times\delta\bfE^s[\bfV]\\
&\qquad\qquad\qquad\qquad\qquad\qquad
+\delta\bfn[\bfV]\times\frac{\partial\bfE_0}{\partial\bfV'}
+\delta\bfn[\bfV']\times\frac{\partial\bfE_0}{\partial\bfV}
\bigg)\\
&\qquad-\bfn\times\frac{\partial\big(\delta\bfE^s[\bfV']\big)}{\partial\bfV}
-\bfn\times\frac{\partial\big(\delta\bfE^s[\bfV]\big)}{\partial\bfV'}
-\bfn\times\frac{\partial^2\bfE_0}{\partial\bfV'\partial\bfV}.
\end{align*}
It remains to show that $\bfga_0^+$ acts as the identity on its argument, i.e., we have
to show that the normal component of its argument vanishes. Therefore, exploiting
\begin{align}\label{eq:crossscalformula}
\langle\bfb\times\bfc,\bfa\rangle\bfa=(\bfa\times\bfb)\times(\bfa\times\bfc),
\end{align}
we compute
\[
\langle\delta\bfn[\bfV]\times\delta\bfE^s[\bfV'],\bfn\rangle\bfn
=
(\bfn\times\delta\bfn[\bfV])\times(\bfn\times\delta\bfE^s[\bfV'])
=
(\bfn\times\delta\bfn[\bfV])\times\bfg^{(\delta,s)}[\bfV'].
\]
\eqref{eq:crossscalformula} yields
\begin{align}\label{eq:normalcancels}
\langle\delta\bfn[\bfV]\times\delta\bfE^s[\bfV'],\bfn\rangle\bfn
=
\langle\underbrace{\bfn\times\bfE_0}_{=0},\delta\bfn[\bfV]\rangle\delta\bfn[\bfV]
-
\bigg\langle\delta\bfn[\bfV]\times\frac{\partial\bfE_0}{\partial\bfV'},\bfn\bigg\rangle\bfn
=
-
\bigg\langle\delta\bfn[\bfV]\times\frac{\partial\bfE_0}{\partial\bfV'},\bfn\bigg\rangle\bfn.
\end{align}
Thus, the normal components of $\delta\bfn[\bfV]\times\delta\bfE^s[\bfV']$ and
$\delta\bfn[\bfV]\times\partial\bfE_0/\partial\bfV'$ cancel. This yields the assertion.
\end{proof}

For numerical computations we have to reformulate the boundary conditions of the previous theorem.
For simplicity, we restrict ourselves to vector fields of the form $\bfV(\omega)=r(\omega)\bfn$ and
$\bfV'(\omega)=r'(\omega)\bfn$ and indicate the dependence on $r$ and $r'$ by $[r]$ and $[r']$, rather
than $[\bfV]$ and $[\bfV']$.
\begin{theorem}\label{thm:main}
Let $\Gamma$ be a boundary of class $C^{3,1}$ and
$\bfV(\omega)=r(\omega)\bfn$, $\bfV'(\omega)=r'(\omega)\bfn$ such that
$\bfV,\bfV'\in C^{3,1}(\Gamma;\R^3)$.
Then, the boundary values of the second derivative in directions
$\bfV,\bfV'$ is a solution of \eqref{eq:secondlocaldu}
with
\begin{align*}
\bfg^{(\delta^2,s)}[r,r']
={}&
-
\big(2H\bfI-\big[\bgrad_{\Gamma}\bfn\big]\big)
\bigg(\kappa^{-2}\div_{\Gamma}(\bfga_{N}^{+}\bfE_0)\Big(
r\big[\bgrad_{\Gamma}r'\big]
+
r'\big[\bgrad_{\Gamma}r\big]
\Big)\times\bfn\\
&\qquad\qquad\qquad\qquad\qquad\qquad+
\Big(
r(\delta\bfE^s[r']\times\bfn)+
r'(\delta\bfE^s[r]\times\bfn)
\Big)\bigg)\\
&\qquad
-
r\bfga_N^+\delta\bfE^s[r']\times\bfn
-
r'\bfga_N^+\delta\bfE^s[r]\times\bfn\\
{}&\qquad
+
\bcurl_{\Gamma}\big(r\langle\delta\bfE^s[r'],\bfn\rangle +r'\langle\delta\bfE^s[r],\bfn\rangle\big)\\
&\qquad+
rr'\Big(
2\kappa^2\bfn\times\big[\bgrad_{\Gamma}\bfn\big]\bgrad_{\Gamma}\div_{\Gamma}\big(\bfga_{N}^+\bfE_0)\\
&\qquad\qquad\qquad+
2\kappa^{-2}\div_{\Gamma}\big(\bfga_{N}^+\bfE_0)\bfn\times\bgrad_\Gamma H
+2H\big(\bfga_N^+\bfE_0\times\bfn-\bcurl_{\Gamma}(\langle\bfE_0,\bfn\rangle)\big)\Big).
\end{align*}
Here, $2H=\tr[\bgrad_{\Gamma}\bfn]$ denotes the mean curvature.
\end{theorem}
\begin{proof}
We reformulate the expressions from theorem~\ref{thm:main}.
\begin{itemize}
\item
Lemma~\ref{lem:nder} and a short computation imply immediately that
\[
\delta^2\bfn[r,r']
=
\big[\bgrad_{\Gamma}\bfn\big]\Big(
r\big[\bgrad_{\Gamma}r'\big]
+
r'\big[\bgrad_{\Gamma}r\big]
\Big)
-
\big\langle
\bgrad_{\Gamma}(r\bfn)
,
\bgrad_{\Gamma}(r'\bfn)
\big\rangle
\bfn.
\]
Combining $\bfn\times\bfE_0=0$ with \eqref{eq:Enrep} yields
\[
\bfE_0=\langle\bfn,\bfE_0\rangle\bfn = -\kappa^{-2}\div_{\Gamma}(\bfga_{N}^{+}\bfE_0)\bfn
\]
and thus
\[
\delta^2\bfn[r,r']\times\bfE_0
=
-\kappa^{-2}\div_{\Gamma}(\bfga_{N}^{+}\bfE_0)
\big[\bgrad_{\Gamma}\bfn\big]
\Big(
r\big[\bgrad_{\Gamma}r'\big]
+
r'\big[\bgrad_{\Gamma}r\big]
\Big)\times\bfn.
\]
\item
We next deal with
\[
\delta\bfn[r]\times\delta\bfE^s[r']
+
\delta\bfn[r]\times\frac{\partial\bfE_0}{\partial (r'\bfn)}.
\]
Therefore, we remark that
\begin{align}\label{eq:dnVxdEV}
\delta\bfn[r]\times\delta\bfE^s[r']
=
\delta\bfn[r]\times\bfga_0^+\big(\delta\bfE^s[r']\big)
+\delta\bfn[r]\times\bfn\langle\delta\bfE^s[r'],\bfn\rangle
\end{align}
and
\begin{align}\label{eq:dndEdV}
\delta\bfn[r]\times\frac{\partial\bfE_0}{\partial (r\bfn')}
=
r'\delta\bfn[r]\times\frac{\partial\bfE_0}{\partial\bfn}
=
r'\delta\bfn[r]\times\bfga_0^+\bigg(\frac{\partial\bfE_0}{\partial \bfn}\bigg)
+
r'\delta\bfn[r]\times\bfn\bigg\langle\frac{\partial\bfE_0}{\partial\bfn},\bfn\bigg\rangle.
\end{align}
In both equations, the normal component is given by the first term of the right-hand side.
Thus, the sum of \eqref{eq:dnVxdEV} and \eqref{eq:dndEdV} has a vanishing normal component
due to \eqref{eq:normalcancels} and is given by
\[
\delta\bfn[r]\times\delta\bfE^s[r']
+
\delta\bfn[r]\times\frac{\partial\bfE_0}{\partial (r'\bfn)}
=
\delta\bfn[r]\times\bfn\langle\delta\bfE^s[r'],\bfn\rangle
+
r'\delta\bfn[r]\times\bfn\bigg\langle\frac{\partial\bfE_0}{\partial \bfn},\bfn\bigg\rangle.
\]
Due to
\[
\div\bfu = \div_{\Gamma}\bfu + 2H\langle\bfu,\bfn\rangle+\bigg\langle\frac{\partial\bfu}{\partial\bfn},\bfn\bigg\rangle,
\]
see \cite[Eq.~(2.5.210)]{Ned01}, it holds
\[
\bfn\bigg\langle\frac{\partial\bfE_0}{\partial \bfn},\bfn\bigg\rangle
=
\bfn\frac{\partial}{\partial\bfn}\langle\bfE_0,\bfn\rangle
=
(\underbrace{\div\bfE_0}_{=0})\bfn-\div_{\Gamma}\big(\underbrace{\bfga_0^+\bfE_0}_{=0}\big)\bfn+2H\langle\bfE_0,\bfn\rangle\bfn
=
2H\langle\bfE_0,\bfn\rangle\bfn,
\]
which together with \eqref{eq:Enrep} yields
\begin{align}\label{eq:intermres}
\delta\bfn[r]\times\delta\bfE^s[r']
+
\delta\bfn[r]\times\frac{\partial\bfE_0}{\partial (r'\bfn)}
=
\delta\bfn[r]\times\bfn\langle\delta\bfE^s[r'],\bfn\rangle
-
2H\kappa^{-2}\div_{\Gamma}(\bfga_{N}^{+}\bfE_0)r'\delta\bfn[r]\times\bfn.
\end{align}
\item
We proceed with
\begin{align*}
\bfn\times\frac{\partial(\delta\bfE^s[r'])}{\partial (r\bfn)}
+
\delta\bfn[r]\times\bfn\langle\delta\bfE^s[r'],\bfn\rangle,
\end{align*}
where the first term occurs in \eqref{eq:secondgeneral} and the second in \eqref{eq:intermres}. Due to
\begin{align}\label{eq:curlformula}
\bcurl\bfu = (\curl_\Gamma\bfu)\bfn+\bcurl_\Gamma(u|_\Gamma\cdot\bfn)-([\bgrad_\Gamma\bfn]\bfu)\times\bfn+\bfn\times\frac{\partial\bfu}{\partial\bfn},
\end{align}
see, e.g. \cite[Eq.~(2.5.211)]{Ned01}, the first term can be rewritten as
\[
\bfn\times\frac{\partial(\delta\bfE^s[r'])}{\partial (r\bfn)}
=
r\bfga_N^+\big(\delta\bfE^s[r']\big)\times\bfn
-
r\bcurl_{\Gamma}\big(\langle\delta\bfE^s[r'],\bfn\rangle\big)
+
r\big([\bgrad_{\Gamma}\bfn]\delta\bfE^s[r']\big)\times\bfn,
\]
where the last term can be reformulated to
\[
r\big([\bgrad_{\Gamma}\bfn]\delta\bfE^s[r']\big)\times\bfn
=
r\big(2H\bfI-[\bgrad_{\Gamma}\bfn]\big)(\delta\bfE^s[r']\times\bfn),
\]
due to the formula $(\bfA\bfb)\times\bfc+\bfb\times(\bfA\bfc)=(\tr(\bfA)\bfI-\bfA^{\intercal})(\bfb\times\bfc)$, $[\bgrad_{\Gamma}\bfn]=[\bgrad_{\Gamma}\bfn]^{\intercal}$, and $\tr[\bgrad_{\Gamma}\bfn]=2H$.

Moreover, since $\delta\bfn[r]=-[\bgrad_{\Gamma}(r\bfn)]\bfn$ due to lemma~\ref{lem:nder} and
\[
-[\bgrad_{\Gamma}(r\bfn)]\bfn\times\bfn - \underbrace{\big[\bgrad_{\Gamma}\bfn\big]r\bfn}_{=0}\times\bfn
=
-\bcurl_{\Gamma}\big(\langle r\bfn,\bfn\rangle\big)
=
-\bcurl_{\Gamma}(r),
\]
it holds
\[
\delta\bfn[r]\times\bfn\langle\delta\bfE^s[r'],\bfn\rangle = -\bcurl_{\Gamma}(r)\langle\delta\bfE^s[r'],\bfn\rangle.
\]
This yields
\begin{align*}
&\bfn\times\frac{\partial(\delta\bfE^s[r'])}{\partial (r\bfn)}
+
\delta\bfn[r]\times\bfn\langle\delta\bfE^s[r'],\bfn\rangle\\
&\qquad\qquad =r\bfga_N^+\big(\delta\bfE^s[r']\big)\times\bfn
-\bcurl_{\Gamma}\big(r\langle\delta\bfE^s[r'],\bfn\rangle\big)
+
r\big(2H\bfI-[\bgrad_{\Gamma}\bfn]\big)(\delta\bfE^s[r']\times\bfn).
\end{align*}
\item
Lastly, we have to deal with
\begin{align}\label{eq:tangentialsecond}
\bfn\times\frac{\partial^2\bfE_0}{\partial(r'\bfn)\partial(r\bfn)}
=
rr'\bfn\times\frac{\partial^2\bfE_0}{\partial\bfn^2}.
\end{align}
Therefore, we recall the formula
\begin{align*}
\Delta \psi
= \Delta _{\Gamma}\psi
+ 2H\frac{\partial\psi}{\partial\mathbf{n}}
+ \frac{\partial ^2\psi}{\partial\mathbf{n}^2}\quad\text{on}~\Gamma,
\end{align*}
cf.~\cite[Eq.~(2.3.212)]{Ned01}. Applied to each component of the $\bfE_0$ this yields
\begin{align}\label{eq:secondnormal}
\frac{\partial^2\bfE_0}{\partial\bfn^2}
=
\bDelta\bfE _0-\bDelta_{\Gamma}\bfE_0-2H\frac{\partial\bfE_0}{\partial\bfn},
\end{align}
with the compontent-wise application of the Laplace operator $\bDelta$ and the Laplace-Beltrami operator
$\bDelta_{\Gamma}$. The first term on the right-hand side of \eqref{eq:secondnormal} can be reformulated as
\[
\bDelta\bfE_0=\bcurl\bcurl\bfE_0+\bgrad\underbrace{\div\bfE_0}_{=0}=\bcurl\bcurl\bfE_0=\kappa^2\bfE_0.
\]
For the second term, we invoke the formula
\[
\bDelta_{\Gamma}\bfn = \Big(\big(2H)^2-2K\Big)\bfn+2\bgrad_{\Gamma}H,
\]
which can be derived from the Weingarten and the Codazzi equations, see
\cite[Chapters 13.3 and 19.3]{Gra06}.
Here, $K=\det\big[\bgrad_{\Gamma}\bfn\big]$ denotes the Gaussian mean curvature. This yields
\begin{align*}
\bDelta_{\Gamma}\bfE_0
={}&
\bDelta_{\Gamma}\big(\langle\bfE_0,\bfn\rangle\big)\bfE_0\\
={}&
(\bDelta_{\Gamma}\langle\bfE_0,\bfn\rangle\big)\bfn
+
2\big[\bgrad_{\Gamma}\bfn\big]\bgrad_{\Gamma}\langle\bfE_0,\bfn\rangle
+
\langle\bfE_0,\bfn\rangle\bDelta_{\Gamma}\bfn\\
={}&
(\bDelta_{\Gamma}\langle\bfE_0,\bfn\rangle\big)\bfn
-
2\kappa^{-2}\big[\bgrad_{\Gamma}\bfn\big]\bgrad_{\Gamma}\div_{\Gamma}\big(\bfga_{N}^+\bfE_0)\\
&\qquad\qquad+
\Big(\big(2H)^2-2K\Big)\langle\bfE_0,\bfn\rangle\bfn-2\kappa^{-2}\div_{\Gamma}\big(\bfga_{N}^+\bfE_0)\bgrad_\Gamma H.
\end{align*}
Since $\bfn\times\bfE_0$ vanishes and since
\[
\bfn\times\frac{\partial\bfE_0}{\partial\bfn}
=
-\bfga_N^+\bfE_0\times\bfn+\bcurl_{\Gamma}(\langle\bfE_0,\bfn\rangle)
\]
due to \eqref{eq:curlformula}, \eqref{eq:tangentialsecond} can be rewritten as
\begin{align*}
\bfn\times\frac{\partial^2\bfE_0}{\partial(r'\bfn)\partial(r\bfn)}
={}&
-rr'\Big(
2\kappa^{-2}\bfn\times\big[\bgrad_{\Gamma}\bfn\big]\bgrad_{\Gamma}\div_{\Gamma}\big(\bfga_{N}^+\bfE_0)\\
&\qquad\qquad+
2\div_{\Gamma}\big(\bfga_{N}^+\bfE_0)\bfn\times\bgrad_\Gamma H
+2H\big(\bfga_N^+\bfE_0\times\bfn-\bcurl_{\Gamma}(\langle\bfE_0,\bfn\rangle)\big)\Big).
\end{align*}
\end{itemize}
Putting all expressions together yields the assertion.
\end{proof}
We note the immediate connection to the boundary values of the first local shape derivative which is useful for
the purposes of implementation.
\begin{corollary}
Let the assumptions of theorem~\ref{thm:main} hold.
Given an electric field $\bfE$, let
\[
\bfg_{\bfE}^{(\delta,s)}[\bfV]
=
-
\langle\bfV,\bfn\rangle\bfga_{N}^{+}\bfE\times\bfn
+
\bcurlg\big(\langle\bfV,\bfn\rangle\langle\bfE,\bfn\rangle\big)
\]
be defined similarly to the boundary values of the first local shape derivative from theorem~\ref{thm:firstlocal}. Then it holds
\begin{align*}
\bfg^{(\delta^2,s)}[r,r']
={}&
-
\big(2H\bfI-\big[\bgrad_{\Gamma}\bfn\big]\big)
\bigg(
\kappa^{-2}\div_{\Gamma}(\bfga_{N}^{+}\bfE_0)\Big(r\big[\bgrad_{\Gamma}r'\big]+
r'\big[\bgrad_{\Gamma}r\big]
\Big)\times\bfn\\
&\qquad\qquad\qquad\qquad\qquad\qquad\qquad\qquad\qquad\quad+\Big(
r\bfg_{\bfE_0}^{(\delta,s)}[r']+
r'\bfg_{\bfE_0}^{(\delta,s)}[r]
\Big)\bigg)\\
&\qquad
+\bfg_{\delta\bfE^s[r']}^{(\delta,s)}[r]+\bfg_{\delta\bfE^s[r]}^{(\delta,s)}[r']\\
&\qquad+
rr'\Big(
2\kappa^{-2}\bfn\times\big[\bgrad_{\Gamma}\bfn\big]\bgrad_{\Gamma}\div_{\Gamma}\big(\bfga_{N}^+\bfE_0)\\
&\qquad\qquad\qquad\qquad+
2\kappa^{-2}\div_{\Gamma}\big(\bfga_{N}^+\bfE_0)\bfn\times\bgrad_\Gamma H
-2H\bfg_{\bfE_0}^{(\delta,s)}[1]\Big).
\end{align*}
\end{corollary}
\begin{corollary}
Under the assumptions of the previous corollary it holds
\begin{align*}
\bfg^{(\delta^2,s)}[r,r]
={}&
-2\big(2H\bfI-\big[\bgrad_{\Gamma}\bfn\big]\big)
\Big(
\kappa^{-2}\div_{\Gamma}(\bfga_{N}^{+}\bfE_0)r\big[\bgrad_{\Gamma}r\big]
\times\bfn+
r\bfg_{\bfE_0}^{(\delta,s)}[r]
\Big)\\
&\qquad+2\bfg_{\delta\bfE^s[r]}^{(\delta,s)}[r]\\
&\qquad+
r^2\Big(
2\kappa^{-2}\bfn\times\big[\bgrad_{\Gamma}\bfn\big]\bgrad_{\Gamma}\div_{\Gamma}\big(\bfga_{N}^+\bfE_0)\\
&\qquad\qquad\qquad\qquad+
2\kappa^{-2}\div_{\Gamma}\big(\bfga_{N}^+\bfE_0)\bfn\times\bgrad_\Gamma H
-2H\bfg_{\bfE_0}^{(\delta,s)}[1]\Big).
\end{align*}
\end{corollary}

\section{Mean of the Scattered Field on Random Domains}
\label{sec:BIE}

\subsection{Second Derivative of the Mean}
For the following considerations, we shall assume that the random perturbation
fields perturb the geometry in normal direction only, i.e., we assume
$\bfV(\omega)=r(\omega)\bfn$ and $\bfV'(\omega)=r(\omega)\bfn'$. Moreover, without loss of generality, 
we assume that the boundary perturbations are centered, i.e.,
\[
\E[r]=0,\quad\text{and thus formally}\quad \E[D_{\varepsilon}^c] = D_0^c.
\]
This is not a restriction, since one can easily recenter the random 
field by considering $\tilde{r}(\omega)=r(\omega)-\mathbb{E}[r]$.
The assumption $\E[r]=0$ especially means that the mean of the first order local shape 
derivative \eqref{eq:firstlocaldu} vanishes due to
\[
\E\big[\bfg^{(\delta,s)}[\bfV]\big]
=
-
\underbrace{\E[r]}_{=0}\bfga_{N}^{+}\bfE_0\times\bfn
+
\bcurlg\big(\underbrace{\E[r]}_{=0}\langle\bfE_0,\bfn\rangle\big)=\bfNull.
\]
Using \eqref{eq:correctionterm} and taking the mean of \eqref{eq:shapeTaylor}
this implies
\begin{align}\label{eq:mean}
\E[\bfE_{\varepsilon}^s](\mathbf{x}) = \bfE_0^s(\mathbf{x}) + \frac{\varepsilon^2}{2}
\E\big[\delta^2\bfE^s[r,r]\big](\mathbf{x})+\mathcal{O}(\varepsilon ^3),
\quad\mathbf{x}\in G,
\end{align}
for $0<\varepsilon\leq\varepsilon_0$.

The asymptotic expansion for $\mathbb{E}[\bfE_{\varepsilon}^s]$ can, under certain circumstances, be improved to fourth
order accuracy with help of the following lemma, which is inspired by
\cite[Lemma 2.3]{HP18} and \cite[Chapter XI]{Loe78}.

\begin{lemma}
Assume that the boundary perturbations in normal direction are given by an
expansion
\[
r(\omega)
=
\sum _{i=1}^M\kappa _i(\mathbf{x})X_i(\omega),
\]
where $X_i$, $i=1,\ldots,M$, are independent and identically distributed
random variables. Then, it holds
\[
\delta^3 \bfE^s[
r(\omega),
r(\omega),
r(\omega)
]
=
\sum _{i,j,k=1}^M\delta^3 \bfE^s[\kappa _i,\kappa _j,
\kappa _k]
X_i(\omega)X_j(\omega)X_k(\omega),
\]
provided that the \emph{third order local shape derivative} $\delta^3\bfE$,
as usual given as the local shape derivative of the second order local shape derivative, exists.
\end{lemma}
\begin{proof}
The assertion follows directly by exploiting the trilinearity of $\delta^3\bfE$.
\end{proof}
Obviously, due to the independence of the random 
variables $(X_i)_{i=1}^M$, it holds
\[
\mathbb{E}[\delta ^3\bfE^s]=\sum_{i=1}^M\delta^3 
\bfE^s[\kappa _i\cdot{\bf n},\kappa _i\cdot{\bf n},\kappa _i\cdot{\bf n}]
\mathbb{E}[X_i^3]=0,
\]
if the probability distribution of the $X_i$ is symmetric around zero.
The expansion of the 
mean $\mathbb{E}[\bfE_{\varepsilon}^s]$ for \eqref{eq:mean} is thus fourth
order accurate if we assume that the third order local shape 
derivative is Lipschitz-continuous.

The following theorem is an immediate consequence of taking the mean and interchanging
integration and differentiation.
\begin{theorem}
	Under the same assumptions as in theorem~\ref{thm:main} it holds
	\begin{align}\label{eq:meansecondlocaldu}
	\begin{aligned}
	\big(\bcurl\bcurl-\kappa^2\big)\E\big[\delta^2\bfE^s[r,r]\big] ={}& \bfNull&&\text{in}~D_0^c,\\
	\E\big[\delta^2\bfE^s[r,r]\big] ={}&\E\big[\bfg^{(\delta^2,s)}[r,r]\big]
	&&\text{on}~\Gamma,
	\end{aligned}
	\end{align}
	with
	\begin{align}\label{eq:meansecondlocaldug}
	\begin{aligned}
	\E\big[\bfg^{(\delta^2,s)}[r,r]\big]
	={}&
	-2\big(2H\bfI-\big[\bgrad_{\Gamma}\bfn\big]\big)
	\Big(
	\kappa^{-2}\div_{\Gamma}(\bfga_{N}^{+}\bfE_0)\E\Big[r\big[\bgrad_{\Gamma}r\big]\Big]
	\times\bfn+
	\E\big[r\bfg_{\bfE_0}^{(\delta,s)}[r]\big]
	\Big)\\
	&\qquad+2\E\big[\bfg_{\delta\bfE^s[r]}^{(\delta,s)}[r]\big]\\
	&\qquad+
	\E[r^2]\Big(
	2\kappa^{-2}\bfn\times\big[\bgrad_{\Gamma}\bfn\big]\bgrad_{\Gamma}\div_{\Gamma}\big(\bfga_{N}^+\bfE_0)\\
	&\qquad\qquad\qquad\qquad+
	2\kappa^{-2}\div_{\Gamma}\big(\bfga_{N}^+\bfE_0)\bfn\times\bgrad_\Gamma H
	-2H\bfg_{\bfE_0}^{(\delta,s)}[1]\Big).
	\end{aligned}
	\end{align}
\end{theorem}
The boundary values for the second order correction term  \eqref{eq:meansecondlocaldug} require the computation
of the entities
\[
\E\Big[r\big[\bgrad_{\Gamma}r\big]\Big],\qquad
\E\Big[r\bfg_{\bfE_0}^{(\delta,s)}[r]\Big],\qquad
\E\Big[\bfg_{\delta\bfE^s[r]}^{(\delta,s)}[r]\Big],\qquad
\E[r^2],
\]
which depend non-linearly on the random perturbations $r$. As we will see,
we can compute these quantities deterministically, if we reformulate them
as the diagonal of some correlations living in the higher-dimensional product
domain $\Gamma\times\Gamma$. More precisely, it holds
\begin{align}\label{eq:diagonals}
\begin{aligned}
\E\Big[r\big[\bgrad_{\Gamma}r\big]\Big]
={}&
\diag\big((\Id\otimes\bgrad_{\Gamma})\Cor[r]\big),&
\E\Big[\bfg_{\delta\bfE^s[r]}^{(\delta,s)}[r]\Big]
={}&
\diag(\bfA),
\\
\E[r^2]
={}&\diag\big(\Cor[r]\big),&
\E\Big[r\bfg_{\bfE_0}^{(\delta,s)}[r]\Big]
={}&\diag(\bfB),
\end{aligned}
\end{align}
with
\begin{align}\label{eq::AB}
\begin{aligned}
\bfA ={}& 
-
\big(\Id\otimes(\cdot\times\bfn)\big)
\Cor\big[r,\bfga_{N}^{+}\delta\bfE^s[r]\big]
+
\kappa^{-2}(\Id\otimes\bcurlg\div_{\Gamma})
\Cor\big[r,\bfga_{N}^{+}\delta\bfE^s[r]\big],\\
\bfB ={}& -
\big(\Id\otimes(\bfga_{N}^{+}\bfE_0\times\bfn)\big)
\Cor[r]
+
\kappa^{-2}(\Id\otimes\bcurlg)\Big(\big(\Id\otimes(\div_{\Gamma}\bfga_{N}^{+}\bfE_0)\big)
\Cor[r]\Big).
\end{aligned}
\end{align}
The crucial observation which we have to make is that the entities $\bfga_{N}^{+}\bfE_0$ and
$\bfga_{N}^{+}\delta\bfE^s[r]$ can be obtained from the Neumann data of the electric wave
equation with appropriate boundary conditions. Numerically, this can be achieved for example
with finite element or boundary element methods. We shall opt for boundary element methods,
since the corresponding boundary integral equations allow for a more natural treatment of
the arising tensor product equations.

\subsection{Boundary Integral Equations}
We shortly recapitulate the theory of boundary integral equations for the solution of electromagnetic
scattering problems. A comprehensive review covering their theory and numerical discretization
may also be found in \cite{BH03}, whose presentation we will follow closely.

Based on the potentials
\begin{align*}
\Psi_V^\kappa(\phi)(\mathbf{x})
={}&
\int_{\Gamma}\frac{e^{i\kappa\|\mathbf{x}-\mathbf{y}\|}}{4\pi\|\mathbf{x}-\mathbf{y}\|}\phi(\mathbf{y})\d\sigma _\mathbf{y},&&\mathbf{x}\in D_0\cup D_0^c,\\
\boldsymbol\Psi_\mathbf{V}^\kappa(\boldsymbol\mu)(\mathbf{x})
={}&
\int_{\Gamma}\frac{e^{i\kappa\|\mathbf{x}-\mathbf{y}\|}}{4\pi\|\mathbf{x}-\mathbf{y}\|}\boldsymbol\mu(\mathbf{y})\d\sigma _\mathbf{y},&&\mathbf{x}\in D_0\cup D_0^c,
\end{align*}
the \emph{Maxwell single layer potential}
$\boldsymbol\Psi_{SL}^\kappa$ is defined as
\[
\boldsymbol\Psi_{SL}^\kappa(\boldsymbol\mu)
=
\boldsymbol\Psi_{\mathbf{V}}^{\kappa}(\boldsymbol\mu)
+
\frac{1}{\kappa^2}\bgrad\Psi_{V}^{\kappa}(\div_{\Gamma}\boldsymbol\mu)
\qquad
\text{in}~D_0\cup D_0^c,
\]
and the \emph{Maxwell double layer potential} $\boldsymbol\Psi_{DL}^\kappa$ as
\[
\boldsymbol\Psi_{DL}^\kappa(\boldsymbol\mu)
=
\bcurl\boldsymbol\Psi_{\mathbf{V}}^{\kappa}(\boldsymbol\mu)
\qquad
\text{in}~D_0\cup D_0^c.
\]
We note that both operators are continuous with mapping properties
\[
\boldsymbol\Psi_{SL}^\kappa,\boldsymbol\Psi_{DL}^\kappa
\colon
\chispace
\to
\bfH_\loc(\bcurl\bcurl,D_0\cup D_0^c)\cap\bfH_\loc(\div 0,D_0\cup D_0^c),
\]
satisfy
\[
(\bcurl\bcurl-\kappa^2)\boldsymbol\Psi_{SL}^\kappa(\boldsymbol\mu)=0,\qquad
(\bcurl\bcurl-\kappa^2)\boldsymbol\Psi_{DL}^\kappa(\boldsymbol\mu)=0,
\]
and the Silver-M\"uller radiation conditions for all $\boldsymbol{\mu}\in\chispace$.
Using the average operator $\{\gamma\}_{\Gamma}\isdef\frac{1}{2}(\gamma^++\gamma^-)$, we define
boundary integral operators
\[
\bfcS_{\kappa}=\{\bfga_t\}_{\Gamma}\circ\boldsymbol{\Psi}_{SL}^{\kappa},\qquad
\bfcC_{\kappa}=\{\bfga_t\}_{\Gamma}\circ\boldsymbol{\Psi}_{DL}^{\kappa},
\]
which are continuous with mapping properties
\[
\bfcS_{\kappa},\bfcC_{\kappa}\colon\chispace\to\chispace.
\]
We also note that $\bfcS_{\kappa}$ is an isomorphism on $\chispace$ if $\kappa$ is not a non-resonant wavenumber.

Having the basic definitions at hand, we can recall the following two identities
which will be important for the computation of the required correlations in
\eqref{eq::AB}. First, we note that for given boundary values $\bfg\in\chispace$ it holds
\begin{align}\label{eq:EFIE}
\bfcS_{\kappa}\big[\bfga_N\bfE_0^s\big]_{\Gamma}=\bfg\qquad\text{on}~\Gamma,
\end{align}
where we denote by $[\gamma]_{\Gamma}=\gamma^+-\gamma^-$ the jump across $\Gamma$.
For a perfect electric conductor $D_0$, this jump coincides with the Neumann trace of the total electric field, i.e.,
\[
\big[\bfga_N\bfE_0^s\big]_{\Gamma}
=
\bfga_N^+\bfE_0.
\]
Second, given boundary values $\bfg\in\chispace$, we can obtain the Neumann data
$\bfga_N^+\bfE_0^s$ of the solution to \eqref{eq:refmaxwell} by
\begin{align}\label{eq:D2N}
\bfcS_{\kappa}\big(\bfga_N^+\bfE_0^s\big)=\bigg(\frac{1}{2}+\bfcC_{\kappa}\bigg)\bfg\qquad\text{on}~\Gamma.
\end{align}

Depending on which boundary data are available, the scattered electric field in $D_0^c$ may be
expressed as
\[
\bfE_0^s=\boldsymbol\Psi_{SL}^\kappa\Big(\big[\bfga_N\bfE_0^s\big]_{\Gamma}\Big)\qquad\text{in}~D_0^c,
\]
or
\[
\bfE_0^s=-\boldsymbol\Psi_{SL}^\kappa\big(\bfga_N^+\bfE_0^s\big)-\boldsymbol\Psi_{DL}^\kappa\big(\bfga_{t}^+\bfE_0^s\big)\qquad\text{in}~D_0^c,
\]
where the latter is also referred to as Stratton-Chu representation formula.

\subsection{Correlation Calculus}
With the machinery of boundary integral equations at hand, the Neumann data of the total
electric field are given as the solution of
\begin{align}\label{eq:TPeq}
\bfcS_{\kappa}\big(\bfga_{N}^+\bfE_0\big)=-\bfga_{N}^+\bfE^i\qquad\text{on}~\Gamma,
\end{align}
due to \eqref{eq:EFIE}. This allows for the computation of $\bfB$ in \eqref{eq::AB}.
To compute the correlation $\bfA$ from \eqref{eq::AB} we rewrite it as
\begin{align*}
\bfA =
-
\big(\Id\otimes(\cdot\times\bfn)\big)
\tilde{\bfA}
+
\kappa^{-2}(\Id\otimes\bcurlg\div_{\Gamma})
\tilde{\bfA}
\qquad\text{on}~\Gamma\times\Gamma,
\end{align*}
with
\begin{align*}
	\tilde{\bfA} = \Cor\big[r,\bfga_{N}^{+}\delta\bfE^s[r]\big].
\end{align*}
Then, a relation between $\tilde{\bfA}$ and $\bfB$ is directly given by \eqref{eq:D2N}, i.e.,
it holds
\begin{align*}
	\big(\Id\otimes\bfcS_{\kappa}\big)\tilde{\bfA} = \bigg(\Id\otimes\bigg(\frac{1}{2}+\bfcC_{\kappa}\bigg)\bigg)\bfB\qquad\text{on}~\Gamma\times\Gamma.
\end{align*}
The discretization of these equations is the topic of the next subsection.

\section{Galerkin Discretization}
\label{sec:galerkin}

As outlined in the previous section, the computation of the boundary values for the
second order correction term requires the operator equations \eqref{eq:EFIE} and
\eqref{eq:TPeq} to be solved. To improve readability we start with the variational
formulation of \eqref{eq:EFIE}, rather than \eqref{eq:TPeq}. The variational formulations read
\begin{quote}
find $\bfga_{N}^+\bfE_0\in\chispace$ such that
\[
\big\langle\bfcS_{\kappa}\big(\bfga_{N}^+\bfE_0\big),\bfv\big\rangle_{\times}=
\langle\bfg,\bfv\rangle_{\times}
\]
for all $\bfv\in\chispace$,
\end{quote}
for \eqref{eq:EFIE} and
\begin{quote}
find $\bfga_{N}^+\bfE_0^s\in\chispace$ such that
\begin{align}\label{eq:D2Nvar}
\big\langle\bfcS_{\kappa}\big(\bfga_{N}^+\bfE_0^s\big),\bfv\big\rangle_{\times}=
\bigg\langle\bigg(\frac{1}{2}+\bfcC_{\kappa}\bigg)\bfg,\bfv\bigg\rangle_{\times}
\end{align}
for all $\bfv\in\chispace$,
\end{quote}
for \eqref{eq:D2N}, see also \cite{BH03}.
We will use $N_{\times}$-dimensional piecewise polynomial finite element spaces $\bfS_h^d\subset\chispace$ with polynomial degree of at least $d\geq0$,
generated from a quasi-uniform mesh. Since the right-hand side of \eqref{eq:D2Nvar}
is difficult to compute due to the involved integral operator, we use the
approximation
\[
\bigg(\frac{1}{2}+\bfcC_{\kappa}\bigg)\bfg
\approx
\frac{1}{2}\bfg-\bfcC_{\kappa}\bfPi_h^{\bfNull}\Big(\big(\bfPi_h^{\bfNull}(\bfg\times\bfn)\big)\times\bfn\Big)
\]
where $\bfPi_h^{\bfNull}$ denotes the $\bfL^2(\Gamma)$-projection onto $\bfS_h^d$.
This yields the systems of linear equations
\[
\underline{\bfS}_{\kappa}\underline{\bfj}^{(\bfga_{N}^+\bfE_0)}=\underline{\bfg}
\]
for \eqref{eq:EFIE}, and
\[
\underline{\bfS}_{\kappa}\underline{\bfj}^{(\bfga_{N}^+\bfE_0^s)}=\bigg(\frac{1}{2}\underline{\bfI}-\underline{\bfC}_{\kappa}\underline{\bfM}_{\bfNull}^{-1}\underline{\bfM}_{\times}\underline{\bfM}_{\bfNull}^{-1}\bigg)\underline{\bfg}
\]
for \eqref{eq:D2N}. The occurring matrices are defined as
\begin{align*}
\underline{\bfS}_{\kappa}={}&\big[\langle\bfcS_{\kappa}\bfpsi_j,\bfpsi_i\rangle_{\times}\big]_{i,j=1}^{N_\times},&
\underline{\bfM}_{\bfNull}={}&\big[(\bfpsi_j,\bfpsi_i)_{\bfNull}\big]_{i,j=1}^{N_\times},\\
\underline{\bfC}_{\kappa}={}&\big[\langle\bfcC_{\kappa}\bfpsi_j,\bfpsi_i\rangle_{\times}\big]_{i,j=1}^{N_\times},&
\underline{\bfM}_{\times}={}&\big[\langle\bfpsi_j,\bfpsi_i\rangle_{\times}\big]_{i,j=1}^{N_\times},
\end{align*}
for $\bfpsi_i\in\bfS_h^d$, $i=1,\ldots,N_{\times}$.

Unfortunately, $\bfB$ cannot be computed analytically, but has to be
obtained via a numerical approximation. Denoting the $L^2(\Gamma)$-projection onto the $N_{0}$-dimensional, continuous, piecewise polynomial finite element space $S_h^{d+1}\subset L^2(\Gamma)$ with polynomial degree $d+1$ by $\Pi_h^0$, we approximate
\begin{align*}
\bfB \approx \bfB_h={}&-
\Big(\Id\otimes\big((\bfga_{N}^{+}\bfE_0)_h\times\bfn\big)\Big)
\big(\Id\otimes\Pi_h^0\big)\Cor[r]\\
&\qquad\qquad+
\kappa^{-2}(\Id\otimes\bcurlg)\Big(\Big(\Id\otimes\big(\Pi_h^0\div_{\Gamma}(\bfga_{N}^{+}\bfE_0)_h\big)\Big)
\big(\Id\otimes\Pi_h^0\big)\Cor[r]\Big).
\end{align*}
This yields
\begin{align}\label{eq:Bhmatrix}
\underline{\bfB}\approx\underline{\bfB}_h
=
\big(-\underline{\bfN}_1\underline{\bfM}_0^{-1}+\kappa^{-2}\underline{\bfN}_2\underline{\bfM}_0^{-1}\big)\underline{\bfC},
\end{align}
with the matrices
\begin{align*}
\underline{\bfN}_{1}={}&\big[-\langle(\bfga_{N}^{+}\bfE_0)_h\psi_j,\bfpsi_i\rangle_{\bfNull}\big]_{\substack{i=1\ldots N_\times\\j=1,\ldots N_0}},&
\underline{\bfM}_{0}={}&\big[(\psi_j,\psi_i)_{0}\big]_{i,j=1}^{N_0},\\
\underline{\bfN}_{2}={}&\big[\langle\Pi_h^0\div_{\Gamma}(\bfga_{N}^{+}\bfE_0)_h\psi_j,\div_{\Gamma}\bfpsi_i\rangle_{\bfNull}\big]_{\substack{i=1\ldots N_\times\\j=1,\ldots N_0}},&
\underline{\bfC}={}&\big[\langle\Cor[r]\psi_j,\psi_i\rangle_{0}\big]_{i,j=1}^{N_0},
\end{align*}
for $\bfpsi_i\in\bfS_h^d$, $i=1,\ldots,N_{\times}$ and $\psi_i\in S_h^d$, $i=1,\ldots,N_0$. Note that we used the duality
\[
\langle\bcurl_{\Gamma}q,\bfp\rangle_{\times}=\langle q,\div_{\Gamma}\bfp\rangle_0,\qquad~\text{for all}~q\in H^{1/2}(\Gamma),\bfp\in\chispace,
\]
see \cite{BH03}, for the representation of $\underline{\bfN}_2$.

Applying a similar argument as for \eqref{eq:D2Nvar} to the variational formulation for
$\tilde{\bfA}$, i.e.,
\begin{quote}
find $\tilde{\bfA}\in\halfmixchispace$ such that
\[
\big\langle\big(\Id\otimes\bfcS_{\kappa}\big)\tilde{\bfA},\bfv\big\rangle_{0,\times}=
\bigg\langle\bigg(\Id\otimes\bigg(\frac{1}{2}+\bfcC_{\kappa}\bigg)\bigg)\bfB_h,\bfv\bigg\rangle_{0,\times}
\]
for all $\bfv\in\halfmixchispace$,
\end{quote}
yields
\begin{align}\label{eq:tensorprodeq}
\underline{\bfS}_{\kappa}\underline{\tilde{\bfA}}_h\underline{\bfM}_0^{\intercal}=\bigg(\frac{1}{2}\underline{\bfI}-\underline{\bfC}_{\kappa}\underline{\bfM}_{\bfNull}^{-1}\underline{\bfM}_{\times}\underline{\bfM}_{\bfNull}^{-1}\bigg)\underline{\bfB}_h.
\end{align}
Here, we also used the fact that for matrices \({\bf A}\in\mathbb{R}^{k\times n}\), \({\bf B}\in
\mathbb{R}^{\ell\times m}\) and \({\bf X}\in\mathbb{R}^{m\times n}\),
there holds the relation
\begin{align}\label{eq:kronmatvec}
({\bf A}\otimes{\bf B})\tvec({\bf X})=\tvec({\bf BXA}^\intercal),
\end{align}
where, for \({\bf A}=[{\bf a}_1,\ldots,{\bf a}_n]\in\mathbb{R}^{m\times n}\),
the operation $\tvec({\bf A})$ is here defined as
\[
\tvec([{\bf a}_1,\ldots,{\bf a}_n])\isdef
\begin{bmatrix}{\bf a}_1\\ \vdots\\ {\bf a}_n\end{bmatrix}\in\mathbb{R}^{mn}.
\]

With approximations to $\bfA$ and $\bfB$ available, the boundary data of $\E\big[\delta^2\bfE[r,r]\big]$ can now be computed by the use of
\eqref{eq:meansecondlocaldug} and \eqref{eq:diagonals}. Since their computation is a straightforward combination of the arguments in this section, we do not state the precise formula here to keep the exposition simple.

We note that, except for $\underline{\bfN}_{1}$ and $\underline{\bfN}_{2}$, all matrices are available in
standard boundary element codes.
The matrices $\underline{\bfS}_{\kappa}$, $\underline{\bfC}_{\kappa}$ are dense, but
can be efficiently compressed within the $\mathcal{H}$- or $\mathcal{H}^2$-matrix
framework, see \cite{Boe10,Hac15}. $\underline{\tilde{\bfA}}_h$, $\underline{\bfB}_h$, and $\underline{\bfC}$ are also dense matrices, on whose efficient treatment we
will comment in the next section. The other matrices are sparse. The presented discretization scheme can also be applied to the framework of
\cite{JS16}, where the correlation of the first shape derivative is computed.

\section{Efficient Solution of Correlation Equations}
\label{sec:fast}
The efficient solution of correlation equations of the type
\[
(A\otimes B)\Cor[u]=\Cor[f]\qquad\text{in}~H_1\otimes H_2
\]
has been the topic of many articles, see \cite{BG04b,DHS17,HPS12,HSS08b,HZ19,vPS06} to mention
a few. Their commonality is that they are usually employed for the standard Sobolev spaces
$H_1\otimes H_2=H^s\otimes H^t$ or $H_1\otimes H_2=\bfH^s\otimes\bfH^t$.
For the efficient solution of such correlation equations in more involved Sobolev spaces,
sparse edge elements have been proposed in \cite{HJS13}. A general framework for the
construction of sparse finite element spaces for tensor product domains using the
combination technique was proposed in \cite{HPS13}. This framework was also employed for the
construction of sparse boundary element spaces in $\chispace\times\chispace$ in
\cite{JS16}. However, to keep implementation simple, we shall not pursue this approach here,
but rather follow a purely algebraic strategy.

Our starting point is the matrix equation \eqref{eq:tensorprodeq}. As discussed in the
previous section, the matrices $\underline{\tilde{\bfA}}_h$, $\underline{\bfB}_h$, and $\underline{\bfC}$ are dense and thus need to be compressed in some data sparse format. However,
it is well known that the correlation matrix $\underline{\bfC}\in\mathbb{R}^{N_0\times N_0}$ can be well approximated by a low-rank approximation, if the eigenvalues of the
associated integral operator
\[
(\cK\varphi)(\bfx)=\int_{\Gamma}\Cor[r](\bfx,\bfy)\varphi(\bfy)\d\sigma_{\bfy},\quad\bfx\in\Gamma,
\]
decay sufficiently fast. The decay to be expected can be quantified in terms of the
Sobolev smoothness of the prescribed correlation kernel $\Cor[r]$, see \cite{GH17}.
For the computation of such a low-rank approximation, the pivoted Cholesky decomposition,
see \cite{HPS12}, is a simple but effective tool.
Given access to on-the-fly computable matrix entries, the algorithm provides a black-box strategy to obtain an error-controlled
low-rank approximation $\underline{\bfC}\approx\underline{\bfL}\underline{\bfL}^\intercal$,
$\underline{\bfL}\in\mathbb{R}^{N_0\times k}$, $k\ll N_0$. This can be accomplished without
the full assembly of $\underline{\bfC}$. Plugging the low-rank expansion into \eqref{eq:Bhmatrix} and
\eqref{eq:tensorprodeq}, it is straightforward to see that $\underline{\bfB}_h$ and
$\underline{\tilde{\bfA}}_h$ also provide a low-rank structure. Moreover,
these low-rank approximations are straightforwardly computable with standard matrix-vector
products and standard iterative solvers.

Of course, a low-rank approximation of $\Cor[r]$ cannot always be expected, in particular
for correlation functions with with a short correlation length or low Sobolev smoothness when the eigenvalues of $\cK$ decay
slowly. Let us note that the $\cH$-matrix approach to correlation equations, see
\cite{DHP15, DHP17, DHS17}, provides a suitable tool to cope with \eqref{eq:tensorprodeq}
in this case. In particular, all matrices (including the sparse matrices) in \eqref{eq:tensorprodeq} can be represented
in $\cH$-matrix format, see also \cite{DHS17,Hac15}.

\section{Numerical Examples}
\label{sec:examples}
We look at the electromagnetic scattering of an incident wave on the unit sphere.
For the deterministic case, an analytical solution to the scattering problem of a
plane incident wave is given by the Mie series, see, e.g.,
\cite[Chapter 6.9]{Har01}. Thus, when the radius of the sphere is
randomly perturbed according to $U[-\varepsilon,\varepsilon]$, a highly accurate
approximation to the mean can easily be computed by a one dimensional Gaussian
quadrature rule of high order. The reference solution for the following examples will
be computed with 16 quadrature points.

The boundary element computations are performed within the software library Bembel, see \cite{DHM+,DHM+19}. Therefore, the unit sphere is represented with six parametric patches,
on which we use lowest order Raviart-Thomas elements on quadrilaterals for the
discretization of $\chispace$ and linear, patch-wise continuous finite element spaces for the
discretization of $L^2$.
The compression scheme for the boundary element matrices $\underline{\bfC}_\kappa$ and $\underline{\bfS}_\kappa$ was used with a fixed polynomial degree of $15$, see \cite{DKSW18} for details, which is sufficient for usual scattering problems on spheres.
The solution of the matrix equation \eqref{eq:tensorprodeq} is accomplished with the
low-rank approach and the pivoted Cholesky decomposition with a tolerance of $10^{-6}$, see \cite{HPS12}.

Since $U[-\varepsilon,\varepsilon]$ is symmetric around the origin, we can expect a
perturbation error of $\cO(\varepsilon^4)$ when computing the second order correction term, whereas we can expect an error of $\cO(\varepsilon^2)$ when neglecting it.
We exemplarily check the expansions for $\kappa=2$ and compute
the $\ell^\infty$-error on 100 points which are equally distributed on a sphere with radius 2 around the origin.
Figures~\ref{fig:hepsilonasymp} illustrates that the asymptotics $\cO(\varepsilon^2)$ and $\cO(\varepsilon^4)$ are indeed visible in the numerical experiments.
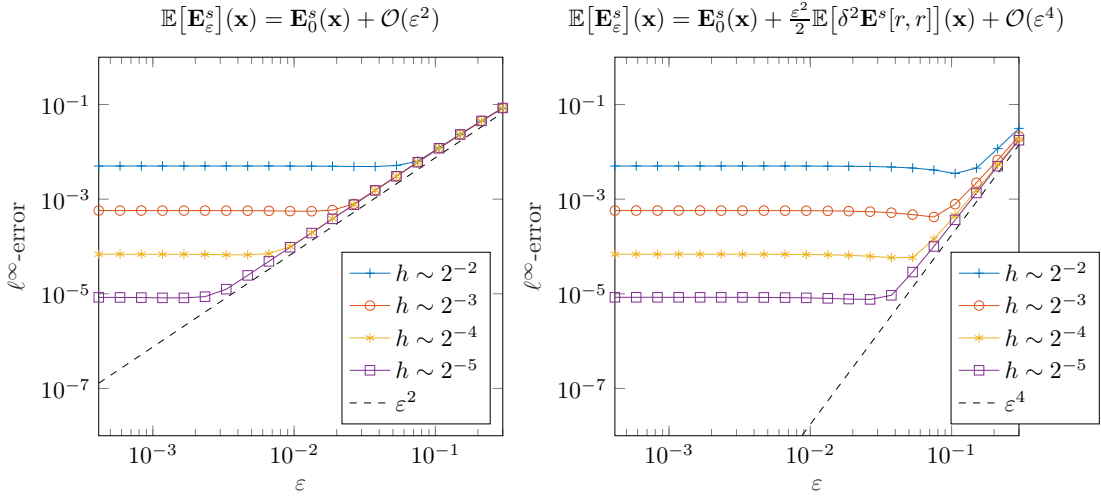
\begin{figure}
	\centering
	\scalebox{0.9}{
%
%
\definecolor{mycolor1}{rgb}{0.00000,0.44700,0.74100}%
\definecolor{mycolor2}{rgb}{0.85000,0.32500,0.09800}%
\definecolor{mycolor3}{rgb}{0.92900,0.69400,0.12500}%
\definecolor{mycolor4}{rgb}{0.49400,0.18400,0.55600}%
\begin{tikzpicture}

\begin{axis}[%
width=0.4\textwidth,
height=0.375\textwidth,
at={(0\textwidth,0\textwidth)},
scale only axis,
xmode=log,
xmin=0.000414,
xmax=0.3,
xminorticks=true,
xlabel style={font=\color{white!15!black}},
xlabel={$\varepsilon$},
ymode=log,
ymin=1e-08,
ymax=1,
yminorticks=true,
ylabel style={font=\color{white!15!black}},
ylabel={$\ell^{\infty}$-error},
axis background/.style={fill=white},
title style={font=\bfseries},
title={$\E\big[\bfE_{{\varepsilon}}^s\big](\mathbf{x}) = \bfE_0^s(\mathbf{x})+\mathcal{O}({\varepsilon}^2)$},
legend style={at={(0.97,0.03)}, anchor=south east, legend cell align=left, align=left, legend plot pos=left, draw=black}
]
\addplot [color=mycolor1, mark=+, mark options={solid, mycolor1}]
  table[row sep=crcr]{%
0.000414	0.00502\\
0.000586	0.00502\\
0.000829	0.00502\\
0.00117	0.00502\\
0.00166	0.00502\\
0.00234	0.00502\\
0.00331	0.00502\\
0.00469	0.00502\\
0.00663	0.00502\\
0.00937	0.00501\\
0.0133	0.00499\\
0.0187	0.00496\\
0.0265	0.0049\\
0.0375	0.00489\\
0.053	0.00514\\
0.075	0.0063\\
0.106	0.0122\\
0.15	0.0235\\
0.212	0.0445\\
0.3	0.0836\\
};
\addlegendentry{$h\sim 2^{-2}$}

\addplot [color=mycolor2, mark=o, mark options={solid, mycolor2}]
  table[row sep=crcr]{%
0.000414	0.000574\\
0.000586	0.000574\\
0.000829	0.000574\\
0.00117	0.000573\\
0.00166	0.000573\\
0.00234	0.000573\\
0.00331	0.000572\\
0.00469	0.00057\\
0.00663	0.000566\\
0.00937	0.000559\\
0.0133	0.000555\\
0.0187	0.000587\\
0.0265	0.000792\\
0.0375	0.00155\\
0.053	0.00307\\
0.075	0.00607\\
0.106	0.012\\
0.15	0.0233\\
0.212	0.0449\\
0.3	0.0844\\
};
\addlegendentry{$h\sim 2^{-3}$}

\addplot [color=mycolor3, mark=asterisk, mark options={solid, mycolor3}]
  table[row sep=crcr]{%
0.000414	6.86000000000001e-05\\
0.000586	6.86000000000001e-05\\
0.000829	6.85e-05\\
0.00117	6.84e-05\\
0.00166	6.82e-05\\
0.00234	6.77e-05\\
0.00331	6.68e-05\\
0.00469	6.66000000000001e-05\\
0.00662999999999999	7.06e-05\\
0.00937	9.96e-05\\
0.0133	0.000195\\
0.0187	0.000387\\
0.0265	0.000768\\
0.0375	0.00153\\
0.053	0.00305\\
0.075	0.00605\\
0.106	0.0119\\
0.15	0.0233\\
0.212	0.045\\
0.3	0.0845000000000001\\
};
\addlegendentry{$h\sim 2^{-4}$}

\addplot [color=mycolor4, mark=square, mark options={solid, mycolor4}]
  table[row sep=crcr]{%
0.000414	8.40000000000001e-06\\
0.000586	8.37e-06\\
0.000829	8.31e-06\\
0.00117	8.2e-06\\
0.00166	8.18999999999999e-06\\
0.00234	8.69e-06\\
0.00331	1.25e-05\\
0.00469	2.45e-05\\
0.00662999999999999	4.84e-05\\
0.00937	9.63000000000001e-05\\
0.0133	0.000192\\
0.0187	0.000383\\
0.0265	0.000764999999999999\\
0.0375	0.00153\\
0.053	0.00304\\
0.075	0.00603999999999999\\
0.106	0.0119\\
0.15	0.0233\\
0.212	0.045\\
0.3	0.0845000000000001\\
};
\addlegendentry{$h\sim 2^{-5}$}

\addplot [color=black, dashed]
  table[row sep=crcr]{%
0.000414	1.2873744e-07\\
0.3	0.0675999999999999\\
};
\addlegendentry{$\varepsilon^2$}

\end{axis}
\end{tikzpicture}%
}
	\scalebox{0.9}{
%
%
\definecolor{mycolor1}{rgb}{0.00000,0.44700,0.74100}%
\definecolor{mycolor2}{rgb}{0.85000,0.32500,0.09800}%
\definecolor{mycolor3}{rgb}{0.92900,0.69400,0.12500}%
\definecolor{mycolor4}{rgb}{0.49400,0.18400,0.55600}%
\begin{tikzpicture}

\begin{axis}[%
width=0.4\textwidth,
height=0.375\textwidth,
at={(0\textwidth,0\textwidth)},
scale only axis,
xmode=log,
xmin=0.000414,
xmax=0.3,
xminorticks=true,
xlabel style={font=\color{white!15!black}},
xlabel={$\varepsilon$},
ymode=log,
ymin=1e-08,
ymax=1,
yminorticks=true,
ylabel style={font=\color{white!15!black}},
ylabel={$\ell^{\infty}$-error},
axis background/.style={fill=white},
title style={font=\bfseries},
title={$\E\big[\bfE_{{\varepsilon}}^s\big](\mathbf{x}) = \bfE_0^s(\mathbf{x})+\frac{{\varepsilon}^2}{2}{\E\big[\delta^2 \bfE^s[r,r]\big]}(\mathbf{x})+\mathcal{O}({\varepsilon}^4)$},
legend style={at={(1.2,0.03)}, anchor=south east, legend cell align=left, align=left, legend plot pos=left, draw=black}
]
\addplot [color=mycolor1, mark=+, mark options={solid, mycolor1}]
  table[row sep=crcr]{%
0.000414	0.00502\\
0.000586	0.00502\\
0.000829	0.00502\\
0.00117	0.00502\\
0.00166	0.00502\\
0.00234	0.00502\\
0.00331	0.00502\\
0.00469	0.00502\\
0.00663	0.00502\\
0.00937	0.00501\\
0.0133	0.00499\\
0.0187	0.00496\\
0.0265	0.0049\\
0.0375	0.00478\\
0.053	0.00455\\
0.075	0.00414\\
0.106	0.00349\\
0.15	0.00453\\
0.212	0.0117\\
0.3	0.0313\\
};
\addlegendentry{$h\sim 2^{-2}$}

\addplot [color=mycolor2, mark=o, mark options={solid, mycolor2}]
  table[row sep=crcr]{%
0.000414	0.000574\\
0.000586	0.000574\\
0.000829	0.000574\\
0.00117	0.000574\\
0.00166	0.000574\\
0.00234	0.000573\\
0.00331	0.000573\\
0.00469	0.000573\\
0.00663	0.000572\\
0.00937	0.00057\\
0.0133	0.000566\\
0.0187	0.000558\\
0.0265	0.000544\\
0.0375	0.000517\\
0.053	0.000471\\
0.075	0.000419\\
0.106	0.000782\\
0.15	0.00222\\
0.212	0.00674\\
0.3	0.0212\\
};
\addlegendentry{$h\sim 2^{-3}$}

\addplot [color=mycolor3, mark=asterisk, mark options={solid, mycolor3}]
  table[row sep=crcr]{%
0.000414	6.86999999999999e-05\\
0.000586	6.86999999999999e-05\\
0.000829	6.86999999999999e-05\\
0.00117	6.86999999999999e-05\\
0.00166	6.86000000000001e-05\\
0.00234	6.86000000000001e-05\\
0.00331	6.86000000000001e-05\\
0.00469	6.84e-05\\
0.00662999999999999	6.82e-05\\
0.00937	6.77e-05\\
0.0133	6.68e-05\\
0.0187	6.51e-05\\
0.0265	6.21e-05\\
0.0375	5.81e-05\\
0.053	5.82e-05\\
0.075	0.000143\\
0.106	0.000454\\
0.15	0.00154\\
0.212	0.00534\\
0.3	0.0184\\
};
\addlegendentry{$h\sim 2^{-4}$}

\addplot [color=mycolor4, mark=square, mark options={solid, mycolor4}]
  table[row sep=crcr]{%
0.000414	8.43e-06\\
0.000586	8.43e-06\\
0.000829	8.43e-06\\
0.00117	8.42e-06\\
0.00166	8.42e-06\\
0.00234	8.41e-06\\
0.00331	8.40000000000001e-06\\
0.00469	8.37e-06\\
0.00662999999999999	8.31e-06\\
0.00937	8.2e-06\\
0.0133	8.01e-06\\
0.0187	7.73e-06\\
0.0265	7.61e-06\\
0.0375	9.28e-06\\
0.053	2.88e-05\\
0.075	0.0001\\
0.106	0.000365\\
0.15	0.00135\\
0.212	0.00498\\
0.3	0.0176\\
};
\addlegendentry{$h\sim 2^{-5}$}

\addplot [color=black, dashed]
  table[row sep=crcr]{%
0.00663000000000001	3.358704139648e-09\\
0.3	0.01408\\
};
\addlegendentry{$\varepsilon^4$}

\end{axis}
\end{tikzpicture}%
}
	\caption{\label{fig:hepsilonasymp}Asymptotics in $\varepsilon$ for the second (left) and the fourth (right) order perturbation approach using lowest order Raviart-Thomas elements and various mesh widths $h$.}
\end{figure}
For a better comparison of the asymptotics in the perturbation size,
figure~\ref{fig:epsiloncomparison} provides a comparison of the second order and the fourth
order approach in the same figure for the most accurate discretizations used in figure
\ref{fig:hepsilonasymp}. It seems that, in this particular case, the higher order perturbation
approach provides an accuracy which is more accurate than the discretization error of the boundary
element scheme, at least up to a perturbation of up to five percent of the radius. After that, the
higher order approach is still significantly more accurate than the second order approach.
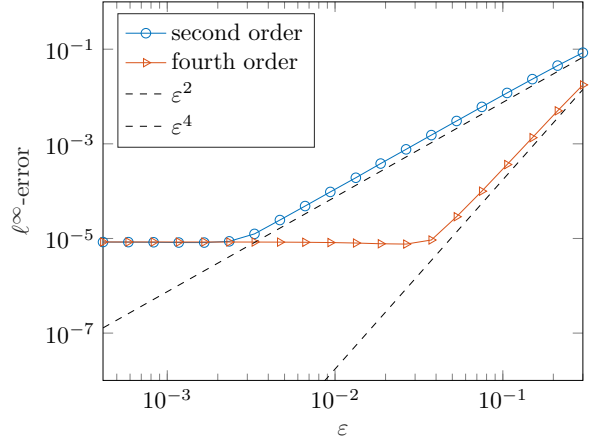
\begin{figure}
	\centering
	\scalebox{0.9}{
%
%
\definecolor{mycolor1}{rgb}{0.00000,0.44700,0.74100}%
\definecolor{mycolor2}{rgb}{0.85000,0.32500,0.09800}%
\begin{tikzpicture}

\begin{axis}[%
width=0.475\textwidth,
height=0.375\textwidth,
at={(0\textwidth,0\textwidth)},
scale only axis,
xmode=log,
xmin=0.000414,
xmax=0.3,
xminorticks=true,
xlabel style={font=\color{white!15!black}},
xlabel={$\varepsilon$},
ymode=log,
ymin=1e-08,
ymax=1,
yminorticks=true,
ylabel style={font=\color{white!15!black}},
ylabel={$\ell^{\infty}$-error},
axis background/.style={fill=white},
title style={font=\bfseries},
title={Comparison of second and fourth order accurate $\varepsilon$-asymptotics},
legend style={at={(0.03,0.97)}, anchor=north west, legend cell align=left, align=left, draw=white!15!black}
]
\addplot [color=mycolor1, mark=o, mark options={solid, mycolor1}]
  table[row sep=crcr]{%
0.000414	8.40000000000001e-06\\
0.000586	8.37e-06\\
0.000829	8.31e-06\\
0.00117	8.2e-06\\
0.00166	8.18999999999999e-06\\
0.00234	8.69e-06\\
0.00331	1.25e-05\\
0.00469	2.45e-05\\
0.00662999999999999	4.84e-05\\
0.00937	9.63000000000001e-05\\
0.0133	0.000192\\
0.0187	0.000383\\
0.0265	0.000764999999999999\\
0.0375	0.00153\\
0.053	0.00304\\
0.075	0.00603999999999999\\
0.106	0.0119\\
0.15	0.0233\\
0.212	0.045\\
0.3	0.0845\\
};
\addlegendentry{second order}

\addplot [color=mycolor2, mark=triangle, mark options={solid, rotate=270, mycolor2}]
  table[row sep=crcr]{%
0.000414	8.43e-06\\
0.000586	8.43e-06\\
0.000829	8.43e-06\\
0.00117	8.42e-06\\
0.00166	8.42e-06\\
0.00234	8.41e-06\\
0.00331	8.40000000000001e-06\\
0.00469	8.37e-06\\
0.00662999999999999	8.31e-06\\
0.00937	8.2e-06\\
0.0133	8.01e-06\\
0.0187	7.73e-06\\
0.0265	7.61e-06\\
0.0375	9.28e-06\\
0.053	2.88e-05\\
0.075	0.0001\\
0.106	0.000365\\
0.15	0.00135\\
0.212	0.00498\\
0.3	0.0176\\
};
\addlegendentry{fourth order}

\addplot [color=black, dashed]
  table[row sep=crcr]{%
0.000414	1.2873744e-07\\
0.3	0.0675999999999999\\
};
\addlegendentry{$\varepsilon^2$}

\addplot [color=black, dashed]
  table[row sep=crcr]{%
0.00663000000000001	3.358704139648e-09\\
0.3	0.01408\\
};
\addlegendentry{$\varepsilon^4$}

\end{axis}
\end{tikzpicture}%
}
	\caption{\label{fig:epsiloncomparison}Comparison of asymptotics in $\varepsilon$ for the second and the fourth order perturbation approach using lowest order Raviart-Thomas elements with mesh width $h\sim 2^{-5}$.}
\end{figure}

In a second numerical experiment, we would like to investigate the dependence of the
perturbation approach on the wavenumber.
Therefore, we observe that the oscillations of the scattered wave increase for larger wavenumbers,
such that the accuracy of the shape Taylor expansions \eqref{eq:shapeTaylor} and \eqref{eq:mean}
is likely to decrease. Thus, we have to expect to lose accuracy when increasing the wavenumber.
Second, since the presented perturbation approach itself relies on
methods for the solution of electromagnetic scattering problems, we can only expect a
well behaved numerical behaviour of the approach if the original scattering problem can
be solved satisfyingly.

\begin{figure}
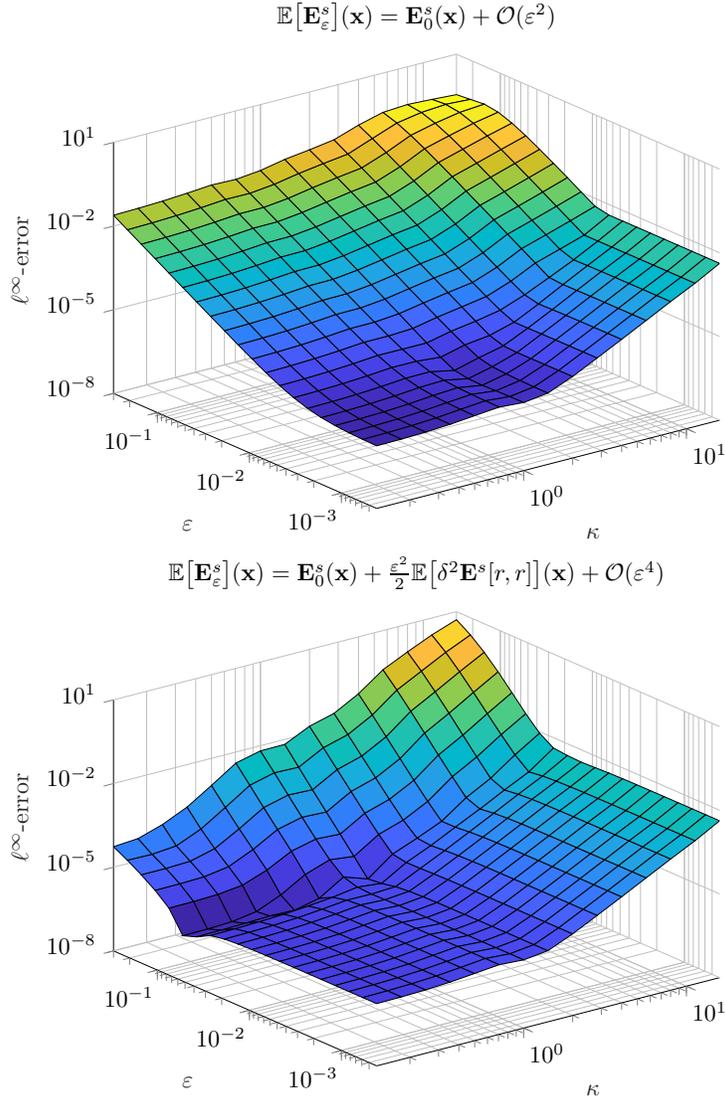

	\centering
	\scalebox{0.9}{\input{eps_vs_kappa_e2o.tex}}
	\scalebox{0.9}{\input{eps_vs_kappa_e3o.tex}}
	\caption{\label{fig:kappaepsilonasymp}Asymptotics in $\varepsilon$ for the second (top) and the fourth (bottom) order perturbation approach using lowest order Raviart-Thomas elements and various wavenumbers $\kappa$ for $h\sim 2^{-5}$.}
\end{figure}
Figure \ref{fig:kappaepsilonasymp} illustrates that we can indeed see these behaviours of the approach.
The decreasing accuracy for wavenumbers $\kappa\geq2$ and small perturbation $\varepsilon$ seems to
be due to the fact that a mesh with mesh size $h\sim 2^{-5}$ cannot resolve the wavenumber anymore.
Thus for $\kappa > 2$, we can not expect an accurate numerical solution of the scattering
problem on the reference domain, which also applies to the second order correction term.

\section{Conclusion}
\label{sec:concl}
We considered time-harmonic electromagnetic scattering problems on perfectly conducting scatterers with uncertain shape. Following the perturbation approach, we improved the existing expansion for the
mean of the scattered electromagnetic field.
Therefore, we characterized the second order correction term $\delta^2\bfE^s[\bfV(\omega),\bfV(\omega)]$
of the shape Taylor expansion
\[
\bfE_\varepsilon^s(\omega) = \bfE^s_0+\varepsilon\delta\bfE^s[\bfV(\omega)]+\frac{\varepsilon^2}{2}\delta^2\bfE^s[\bfV(\omega),\bfV(\omega)]
+\cO(\varepsilon^3)
\]
for a single boundary perturbation $\varepsilon\bfV(\omega)$. The general form $\delta^2\bfE^s[\bfV(\omega),\bfV'(\omega)]$ of this second order correction term
could possibly also be interesting for problems in shape optimization. With
the shape Taylor expansion for a single perturbation available, we characterized the
second order correction term of the mean and obtained an improved expansion
\[
\E\big[\bfE_\varepsilon^s\big]=\bfE^s_0+\frac{\varepsilon^2}{2}
\E\big[\delta^2\bfE^s\big]+\cO(\varepsilon^3).
\]
The error term becomes $\cO(\varepsilon^4)$ if the domain variations are symmetric.
The boundary conditions of the second order correction term can be obtained from the solution
of an additional tensor product problem in the space $\halfmixchispace$.
The formulas for the correction term were validated by numerical experiments and the
accuracy for different mesh widths and wavenumbers was compared for the second order
accurate perturbation approach and the higher order approach. It was clearly visible
that the higher order perturbation approach provides an advantage in terms of accuracy
compared to the second order perturbation approach.

Another use for the second order correction term $\E\big[\delta^2\bfE^s\big]$ which we
did not consider in this article is the
computation of the correlation of the electric field. In \cite{JS16} it was shown that the
covariance of the electric field satisfies
\[
\Cov\big[\bfE_\varepsilon^s\big]=\varepsilon^2\Cor\big[\delta\bfE^s\big]+\cO(\varepsilon^3),
\]
for $0<\varepsilon\leq\varepsilon_0$. Here, the correction term $\Cor[\delta\bfE^s]$ can
be obtained by a tensor product equation in $\chispace\otimes\chispace$ which is very similar
to \eqref{eq:tensorprodeq}. Employing the very same arguments to the correlation yields, see
also \cite{DH18} for the Laplace case,
\[
\Cor\big[\bfE_\varepsilon^s\big]
=
\bfE_0^s\otimes\bfE_0^s
+
\varepsilon^2\Cor\big[\delta\bfE^s\big]
+
\frac{\varepsilon^2}{2}\Big(
\bfE_0^s\otimes\E\big[\delta^2\bfE^s\big]
+
\E\big[\delta^2\bfE^s\big]\otimes\bfE_0^s
\Big)
+
\cO(\varepsilon^3).
\]
Again, both expansions become fourth order accurate if the distribution of the domain
perturbations are symmetric. We note that the low-rank and $\mathcal{H}$-matrix techniques
employed for the solution of \eqref{eq:tensorprodeq} also apply to the equation for
$\Cor\big[\delta\bfE^s\big]$. However, since the application of these techniques to the
remaining equation is straightforward and since the equation itself was thoroughly discussed in \cite{JS16}, we did not include it in our numerical experiments.

\bibliographystyle{plain}
\bibliography{bibl}
\end{document}